%% file: RevisionPrecond.tex
\documentclass[10pt]{article}
\usepackage[usenames,dvipsnames]{xcolor}
\definecolor{OliveGreen}{rgb}{0,0.6,0}
\definecolor{tempblue}{RGB}{36, 56, 231 }
\usepackage{cite}
\usepackage{subfigure}
\usepackage{graphicx, color, graphpap}
\usepackage{amsmath,amssymb,amsthm,bm}
\usepackage{mathabx}
\usepackage{ifthen}
\usepackage{adjustbox}

\usepackage{matlab-prettifier}

\usepackage{mathrsfs}
\usepackage{mathtools}
\usepackage[shortlabels]{enumitem}

\usepackage{physics}
\usepackage{pstricks}
\usepackage{lineno}
\usepackage{float}
\usepackage{fancyhdr}
\usepackage[us,12hr]{datetime} 
\usepackage{exscale}
\usepackage{tabularx}
\usepackage{longtable}
\usepackage{latexsym}
\usepackage{fullpage}       
\usepackage{float}
\usepackage{verbatim}
\usepackage{multirow}
\usepackage{overpic}
\usepackage{comment}

\usepackage{siunitx}
\usepackage{colortbl}

\definecolor{tablegray}{RGB}{215, 219, 221 }

\usepackage{booktabs}
\usepackage{makeidx}
\makeindex

\usepackage{algorithm}   
\usepackage{algpseudocode}

\makeindex

\usepackage{csvsimple}

\usepackage[pagebackref=true]{hyperref}
\hypersetup{
	plainpages=false,       
	unicode=false,          
	pdftoolbar=true,        
	pdfmenubar=true,        
	pdffitwindow=false,     
	pdfstartview={FitH},    
	pdftitle={
 Optimal Diagonal Preconditioning Beyond Worst-Case: Omega Scaling
	           },    
	pdfnewwindow=true,      
	colorlinks=true,        
	linkcolor=blue,         
	citecolor=green,        
	filecolor=magenta,      
	urlcolor=cyan           
}

\numberwithin{equation}{section}  
\numberwithin{table}{section}
\numberwithin{figure}{section}
\numberwithin{algorithm}{section}

\usepackage{mathtools}
\usepackage[noabbrev]{cleveref}

\usepackage{refcount} 

\def\R{\mathbb{R}}

\def\Rpp{\R_{++}}
\def\Sc{\mathbb{S}}

\def\Sn{\Sc^n}

\def\Snp{\Sc_+^n}

\def\Snpp{\Sc_{++}^n}

\def\Rn{\mathbb{R}^n}
\def\Rnm{\mathbb{R}^{n-1}}

\def\Rnp{\mathbb{R}_+^n}
\def\Rnpp{\mathbb{R}_{++}^n}

\newtheorem{theorem}{Theorem}[section]

\newtheorem{linesrch}[theorem]{LineSearch}
\newtheorem{prop}[theorem]{Proposition}

\newtheorem{thm}[theorem]{Theorem}
\newtheorem{cor}[theorem]{Corollary}
\newtheorem{corollary}[theorem]{Corollary}

\newtheorem{remark}[theorem]{Remark}

\newtheorem{lemma}[theorem]{Lemma}

\crefname{thm}{Theorem}{Theorems}
\Crefname{thm}{Theorem}{Theorems}
\crefname{problem}{Problem}{Theorems}
\Crefname{problem}{Problem}{Theorems}
\Crefname{assump}{Assumption}{Theorems}
\crefname{assump}{Assumption}{Theorems}
\Crefname{linesrch}{LineSearch}{Theorems}
\crefname{linesrch}{LineSearch}{Theorems}
\crefname{conjecture}{Conjecture}{Theorems}
\Crefname{conjecture}{Conjecture}{Theorems}
\crefname{prop}{Proposition}{Propositions}
\Crefname{prop}{Proposition}{Propositions}
\crefname{cor}{Corollary}{Corollaries}
\Crefname{cor}{Corollary}{Corollaries}
\crefname{corollary}{Corollary}{Corollaries}
\Crefname{corollary}{Corollary}{Corollaries}
\crefname{lem}{Lemma}{Lemmas}
\Crefname{lem}{Lemma}{Lemmas}
\crefname{lemma}{Lemma}{Lemmas}
\Crefname{lemma}{Lemma}{Lemmas}
\theoremstyle{definition}
\crefname{defn}{definition}{definitions}
\Crefname{defn}{Definition}{Definitions}
\crefname{conj}{Conjecture}{Conjectures}
\Crefname{conj}{Conjecture}{Conjectures}
\crefname{remark}{Remark}{Remarks}
\Crefname{remark}{Remark}{Remarks}
\crefname{rmk}{Remark}{Remarks}
\Crefname{rmk}{Remark}{Remarks}
\crefname{example}{Example}{Examples}
\Crefname{example}{Example}{Examples}
\crefname{algorithm}{Algorithm}{Algorithms}
\Crefname{algorithm}{Algorithm}{Algorithms}
\crefname{align}{}{}
\Crefname{align}{}{}
\crefname{equation}{}{}
\Crefname{equation}{}{}

\newcommand{\textdef}[1]{\textit{#1}\index{#1}}

\newcommand{\cI}{{\mathcal I} }

\newcommand{\cB}{{\mathcal B} }
\newcommand{\cE}{{\mathcal E} }

\newcommand{\cQ}{{\mathcal Q} }
\newcommand{\cM}{{\mathcal M} }
\newcommand{\cR}{{\mathcal R} }

\newcommand{\CV}{\textbf{CV}\,}
\newcommand{\CVp}{\textbf{CV}}

\newcommand{\Mn}{{\mathcal M}^n}

\newcommand{\A}{{\mathcal A}}

\newcommand{\bbm}{\begin{bmatrix}}
\newcommand{\ebm}{\end{bmatrix}}
\newcommand{\bem}{\begin{pmatrix}}
\newcommand{\eem}{\end{pmatrix}}
\newcommand{\beq}{\begin{equation}}
\newcommand{\beqs}{\begin{equation*}}
\newcommand{\bet}{\begin{table}}
\newcommand{\eeq}{\end{equation}}
\newcommand{\eeqs}{\end{equation*}}
\newcommand{\beqr}{\begin{eqnarray}}

\DeclareMathOperator{\cD}{{\mathcal D}_q}
\DeclareMathOperator{\cDo}{{\mathcal D}}
\DeclareMathOperator{\cVo}{{\mathcal V}}

\DeclareMathOperator{\nul}{null}
\DeclareMathOperator{\range}{range}

\DeclareMathOperator{\adj}{{adj}}

\DeclareMathOperator{\Blkdiag}{{Blkdiag}}
\DeclareMathOperator{\diag}{{diag}}

\DeclareMathOperator{\Diag}{{Diag}}

\DeclareMathOperator{\spanl}{{span}}

\DeclareMathOperator{\conv}{{conv}}

\newcommand{\nc}{\newcommand}
\nc{\arrow}{{\rm arrow\,}}
\nc{\Arrow}{{\rm Arrow\,}}
\nc{\BoDiag}{{\rm B^0Diag\,}}
\nc{\bodiag}{{\rm b^0diag\,}}

\nc{\Mm}{{\mathcal M}^{m} }
\nc{\Mmn}{{\mathcal M}^{mn} }
\nc{\Mnr}{{\mathcal M}_{nr} }
\nc{\Mnmr}{{\mathcal M}_{(n-1)r} }
\nc{\kwqqp}{Q{$^2$}P\,}
\nc{\kwqqps}{Q{$^2$}Ps}

\nc{\notinaho}{(X,S)\in \overline{AHO}(\A)}
\nc{\inaho}{(X,S)\in AHO(\A)}

\newcommand{\bea}{\begin{eqnarray}}%
\newcommand{\eea}{\end{eqnarray}}%
\newcommand{\beas}{\begin{eqnarray*}}%
\newcommand{\eeas}{\end{eqnarray*}}%
\newcommand{\Rnn}{\R^{n \times n}}%
\newcommand{\Int}{{\rm int\,}}

\newcommand{\clo}{{\rm cl\,}}
{}

\newcommand{\Hnp}[1][]{\,\mathbb{H}_+^{\ifthenelse{\equal{#1}{}}{n}{#1}}}
\newcommand{\Hn}[1][]{\,\mathbb{H}^{\ifthenelse{\equal{#1}{}}{n}{#1}}}
\newcommand{\Hk}[1][]{\,\mathbb{H}^{\ifthenelse{\equal{#1}{}}{k}{#1}}}
\newcommand{\Dn}[1][]{\,\mathbb{D}^{\ifthenelse{\equal{#1}{}}{n}{#1}}}

\DeclareMathOperator*{\argmin}{argmin}

\title{
\href{http://orion.math.uwaterloo.ca/~hwolkowi/henry/reports/ABSTRACTS.html}{
Optimal Diagonal Preconditioning Beyond Worst-Case Conditioning: Theory
and Practice of Omega Scaling
}
   \footnote{
Emails resp.: sghadimi@uwaterloo.ca, w2jung@uwaterloo.ca,
a3sujana@uwaterloo.ca, david.torregrosa@ua.es,
hwolkowicz@uwaterloo.ca}
   \footnote{
   This report is available at URL:
   \href{https://www.math.uwaterloo.ca/\~hwolkowi/henry/reports/ABSTRACTS.html}{www.math.uwaterloo.ca/\~{ }hwolkowi/henry/reports/ABSTRACTS.html}
   This report is a revision of the unpublished
arXiv \cite{GhadLeoArnDavidWolk:25}.
   }
}

\author{
\href{https://uwaterloo.ca/management-science-engineering/contacts/saeed-ghadimi}{Saeed
Ghadimi}\thanks{Department of Management Science and Engineering, University of
Waterloo, ON, Canada}
\and 
\href{https://uwaterloo.ca/combinatorics-and-optimization/about/people/group/50}{%
Woosuk L. Jung}\thanks{
\href{http://www.math.uwaterloo.ca/co/}{Department of Combinatorics and Optimization}, University of Waterloo,  ON, Canada}
 \and
 \href{https://asujanani6.github.io/}
 {Arnesh Sujanani}\footnotemark[4]
 \and
\href{https://cvnet.cpd.ua.es/curriculum-breve/es/torregrosa-belen-david/110216}{%
David Torregrosa-Bel\'en}\thanks{
\href{https://dmat.ua.es/en/}{Department of Mathematics}, University of Alicante, Alicante, Spain} \and
\href{https://www.math.uwaterloo.ca/~hwolkowi/}
{Henry Wolkowicz}\footnotemark[4]
}

\date{
Revising as of \today, \currenttime
}

\begin{document}


\maketitle

{\bf Key words and phrases:}
$\kappa, \omega$-condition numbers,
diagonal preconditioning, iterative methods, PCG, linear systems,
eigenvalues

{\bf AMS subject classifications:} 15A12, 65F35, 65F08, 65G50,
49J52, 49K10, 90C32

\tableofcontents

\listofalgorithms
\listoftables
\listoffigures

\begin{abstract}
Preconditioning is essential
in many areas of mathematics, and in particular
is a fundamental tool for accelerating
iterative methods for solving linear systems.
In this work, we study optimal diagonal preconditioning
under two distinct notions of conditioning:
the classical worst-case $\kappa$-condition number
and the averaging-based $\omega$-condition number.
We observe that $\omega$-optimal preconditioning generally
outperforms $\kappa$-optimal preconditioning.

For the $\kappa$-optimal preconditioning problem,
we derive an affine-based pseudoconvex
reformulation with three key advantages: all
stationary points are global minima, subgradients
are inexpensive to compute, and the optimization
variable is an $n$-dimensional vector rather
than an $n\times n$ matrix
as in semidefinite programming (SDP)
approaches. 
We develop a simple and highly efficient subgradient method, with convergence guarantees, 
for solving this pseudoconvex formulation. Numerical results indicate that our approach is substantially more 
scalable, efficient, and accurate than existing SDP-based methods, often 
achieving dramatic speedups.

For the $\omega$-condition number,
we provide 
explicit characterizations 
of optimal
diagonal and block diagonal preconditioners.
In particular, we show that
several classical preconditioners, including Jacobi
and row/column normalization, are $\omega$-optimal,
and that matrix balancing schemes 
monotonically
reduce $\omega$ and converge to stationary points
of the two-sided problem.
To the best of our knowledge, this is the first
unified and explicit characterization of 
optimality conditions for both $\kappa$ and $\omega$-based
preconditioning.

Our numerical experiments further reveal a striking 
phenomenon: although $\kappa$-optimal
preconditioners achieve stronger reductions
in the worst-case
condition number,
$\omega$-optimal preconditioners are substantially
cheaper to compute and yield
better performance for iterative
methods such as preconditioned conjugate gradient (PCG) and 
least squares method (LSQR).
Moreover, applying $\omega$-optimal scaling to linear 
systems that are already $\kappa$-optimally
preconditioned leads
to further improvements in PCG iterations.
These results suggest that the
$\omega$-condition number is
more predictive of practical 
solver performance 
and highlight 
the advantages of 
$\omega$-based preconditioning in large-scale settings.

\end{abstract}

\section{Introduction}
\phantomsection
\label{sec:intro}

Preconditioning plays a central role in accelerating iterative methods
for large-scale linear systems, 
e.g.,~\cite{franceschini2018recent,gao2024gradientmethodsonlinescaling,gao2023scalable,doi:10.1287/opre.2022.0592}.
All of these works focus on studying the classical $\kappa$-condition number of a matrix 
$A$ (the ratio of largest to smallest singular values). On the other hand,
another line of research has focused on studying the $\omega$-condition number,
which is the ratio of the arithmetic and geometric mean of singular
values,~\cite{JungTorrWolkOmega2024,DeWo:90}.
Although $\kappa$-optimal preconditioners minimize the worst-case condition
number, our results show that $\omega$-optimal preconditioners are cheaper to
compute and are more predictive of PCG/LSQR performance.

This work is divided into three complementary parts. 
\\(i): First, motivated by the recent work 
in~\cite{gao2023scalable,doi:10.1287/opre.2022.0592} that evaluates
$\kappa$-optimal diagonal preconditioners using convex semidefinite programming
relaxations, we provide a different approach based on
using the nonconvex formulation with $\kappa$. 
In particular, we provide a modified model that
exploits the invariance of eigenvalues for a product of matrices and
solves an essentially unconstrained pseudoconvex minimization problem.
Our approach is significantly
more efficient and robust, often
achieving speedups of over
100x compared with the methods in \cite{gao2023scalable,doi:10.1287/opre.2022.0592}. 
\\(ii): Second, we consider properties of
the $\omega$-condition number.
In particular, we illustrate how to obtain explicit formulae for $\omega$-optimal 
diagonal and block-diagonal preconditioners.
In both settings, we obtain either known
or new preconditioners,
including the Jacobi preconditioner, column and row normalization preconditioners, and those obtained via matrix balancing algorithms such as Sinkhorn-Knopp.
Compared to $\kappa$-optimal preconditioners, the $\omega$-optimal counterparts 
are substantially cheaper 
to compute while retaining strong practical effectiveness.
\\(iii): Finally, in our computational experiments we demonstrate the robustness and efficiency
of our algorithms for finding $\kappa$-optimal preconditioners. We also
demonstrate major advantages of using $\omega$-optimal preconditioners for 
iteratively solving linear systems using preconditioned conjugate gradient (PCG) for positive definite systems
and least squares method (LSQR) for general invertible systems.  In general, we find 
that reductions in the $\omega$-condition number are more strongly correlated with improvements in LSQR and PCG iteration counts than reductions in the $\kappa$-condition number.

For simplicity, we restrict our comparisons to diagonal and block
diagonal preconditioning for large scale sparse
positive definite and general invertible linear systems.
The positive definite case arises from many diverse applications
including: finite element analysis, sparse regression, Newton type
optimization algorithms, and also from
e.g.,~in~\cite{geli:81,gao2023scalable}: (i) the 
so-called normal equations from interior point methods in solving linear
programs and (ii) the Hessians in minimizing logistic regression.

The $\kappa$- and $\omega$-condition numbers for 
the positive definite case $M\succ 0$ are, respectively, 
given by ratios of eigenvalues 
\begin{equation}
\label{eq:kappaomegadef}
\textdef{$\kappa(M) = \frac {\lambda_{\max}(M)}{\lambda_{\min}(M)}$};\quad 
\textdef{$\omega(M) = \frac {\trace(M)/n}{\det(M)^{1/n}}$}= \frac
{\sum_i \lambda_i(M)/n}{\prod_i(\lambda_i(M))^{1/n}}. 
\end{equation}
Preconditioning with $D$ uses
\[
Ax=b \iff D(A(D(D^{-1}x)) = Db,
\]
and reduces a condition number of $DAD$.
For our purposes we exploit the fact that the following functions have
the same values for nonsingular $D$:
\begin{equation}
\label{eq:kappaomegadefequivs}
\kappa(DAD) = \frac {\lambda_{\max}(DAD)}{\lambda_{\min}(DAD)} =
 \frac {\lambda_{\max}(ADD)}{\lambda_{\min}(ADD)};
\qquad 
\omega(DAD) = \frac {\trace(DAD)/n}{\det(DAD)^{1/n}}= 
             \frac {\trace(ADD)/n}{\det(ADD)^{1/n}}. 
\end{equation}
Our algorithms exploit the functions with argument affine in $\bar D=DD$
rather than quadratic in $D$.

In addition, we note that
$\kappa$ and $\omega$ can be considered as \emph{worst-case} and
\emph{average-case} condition numbers, respectively. In fact, it is known that 
the ratio of the arithmetic to geometric mean is approximately the \textdef{coefficient of variation, \CVp}; namely, the ratio between the standard deviation and the mean of a distribution; when
this quantity is close to zero.\footnote{The Young-Trent 
formula~\cite{d687d7d6-f7f0-360e-8186-5a92edcffeb7}
establishes $GM/AM \approx 1 -\frac 12\CVp^2$. Moreover, Taylor's approximation of $1/(1+x)$ around zero yields $AM/GM \approx 1 + \frac 12\CVp^2$.}
And, one of the big advantages of $\omega$ over $\kappa$ is 
that finding explicit formulae for optimal
preconditioners by minimizing $\omega$ with special structure is often
possible due to the analyticity and simplicity of differentiation
of trace and determinant.
\index{$\Snpp$, cone of positive definite matrices}
\index{\CVp, coefficient of variation}

\subsection{Main Contributions}

The main contributions of this paper consist of both theoretical and numerical analyses 
of the $\kappa$ and $\omega$-condition numbers.
First, we provide an affine based
pseudoconvex reformulation of the $\kappa$-optimal diagonal preconditioning problem that
allows for an elegant characterization of its optimality conditions. There
are three advantages of this reformulation: all its stationary points
are global minima, its subgradients are inexpensive to compute, and its optimization variable is 
just an $n\times 1$ vector rather than an $n\times n$ matrix 
as in \textdef{semidefinite programming, SDP}~\cite{doi:10.1287/opre.2022.0592}.
\index{SDP, semidefinite programming} Our extensive numerical experiments
show that subgradient methods based on our reformulation
can effectively converge to a $\kappa$-optimal diagonal preconditioner
more efficiently in time and accuracy than current SDP-based approaches in the
literature \cite{gao2023scalable,doi:10.1287/opre.2022.0592}. 

Second, we provide and exploit the ability to
find explicit formulae for $\omega$-optimal diagonal and block diagonal preconditioners. 
We demonstrate that many popular classic preconditioners,
such as the Jacobi preconditioner and row/column normalization
preconditioners are $\omega$-optimal
preconditioners. We
also display that matrix balancing schemes reduce $\omega$
at every iteration and converge to a stationary point
of the two-sided $\omega$-optimal diagonal preconditioning problem.
To the
best of our knowledge, this is the first time that such a comprehensive
characterization of optimality conditions for both condition numbers is
provided. 

Finally, we conduct additional extensive numerical experiments that
compare the efficiency of the $\kappa$- and $\omega$-condition numbers.
Our results show that since $\omega$-optimal preconditioners typically have
closed-form solutions, they are much
cheaper to construct than $\kappa$-optimal preconditioners. 
Our results further demonstrate that PCG iterations and LSQR
iterations for solving linear systems are, on average, more correlated with
the $\omega$-condition number than the $\kappa$-condition number.

We now continue with the organization of the paper and more details on
the main contributions.
In~\Cref{sect:otptheory}, we present several formulations for
minimizing $\kappa$ along
with the derivatives and optimality conditions. The formulations
are presented to take advantage of the affine approach and then avoid
the positive homogeneity of the problem in order to improve stability of
the problem. In particular, we include~\Cref{thm:optcondkappa} 
that presents a characterization of a 
$\kappa$-optimally diagonally preconditioned $A$ that is based on the
\emph{largest/smallest} eigenpairs. 
This leads to an efficient subgradient algorithm and we
include convergence results. 
In \Cref{sect:optomegadiagproc},
we characterize $\omega$-optimal 
preconditioners with special structures.
A characterization for a $\omega$-optimal
diagonal
preconditioner is
simple as it corresponds to the classical Jacobi 
preconditioner for symmetric positive definite matrices
and the row/column normalization preconditioner 
for general nonsingular matrices. We also show
that the matrix balancing algorithm, Sinkhorn-Knopp,
linearly converges to a stationary point of the two-sided $\omega$-optimal diagonal preconditioning problem.
Finally, we include results on 
extending from diagonal to block
diagonal preconditioning in~\Cref{sect:optblockdiag}. 
We show that the right-sided 
$\omega$-optimal block diagonal preconditioner for
full column rank matrices $A$ comes from 
taking Q-less QR decompositions of the blocks of $A$.
We also give a characterization for the left-sided 
$\omega$-optimal block diagonal preconditioner and 
show that when $A$ is square that it is related
to applying the right-sided optimal preconditioner
to $A^{T}$.


Extensive numerical tests are then conducted in~\Cref{sect:numerics}.
We demonstrate that our subgradient methods are more scalable and efficient 
at computing $\kappa$-optimal
diagonal preconditioners 
than existing SDP-based approaches in the 
literature \cite{gao2023scalable,doi:10.1287/opre.2022.0592}.
Moreover, the resulting preconditioners
yield more 
substantial improvements in PCG performance for
solving linear systems 
compared to those produced by \cite{gao2023scalable,doi:10.1287/opre.2022.0592}.

We further compare optimal
diagonal preconditioning strategies based on the condition numbers
$\kappa$ and $\omega$.
In particular, we observe that the
Sinkhorn-Knopp algorithm, which 
converges to a stationary point of the two-sided $\omega$-optimal preconditioning problem,
produces 
preconditioners that significantly outperform
the two-sided $\kappa$-optimal method 
of \cite{doi:10.1287/opre.2022.0592} in terms of LSQR iteration counts for 
minimizing $\|b-Ax\|$.
While Sinkhorn-Knopp
dramatically reduces $\omega$,
the method of \cite{doi:10.1287/opre.2022.0592} achieves a stronger
reduction in $\kappa$.
This is notable as despite this difference, Sinkhorn-Knopp yields better
LSQR performance on the majority of instances
while also being computationally cheaper.

Finally, we demonstrate
that applying $\omega$-optimal diagonal
scaling to positive definite linear systems
that are already $\kappa$-optimally scaled
leads to significant improvements in PCG 
performance. Overall,
our results highlight two 
key advantages of $\omega$-optimal preconditioners:
they are inexpensive to compute and
substantially reduce the iteration counts of iterative methods.

\subsection{Background; Preliminaries}
\index{$A\in \Mn$ nonsingular}
Given a linear system $Ax=b$ with $A\in \Mn, n\times n$ matrix, nonsingular, 
and $A \approx P_1P_2$, preconditioning effectively solves the given system
by solving the following system for $y$:
\begin{equation}
\label{eq:precondPinvAPinv}
P_1^{-1}A P_2^{-1}y=P_1^{-1}b \text{  for  } y, \,\, 
\,\, x=P_2^{-1}y.
\end{equation}
It is assumed that the solutions with $P_1, P_2$ are
inexpensive.\footnote{In some of our theoretical work 
for notational convenience, we use $D\leftarrow P^{-1}$, below.}
Surveys are given in
e.g.,~\cite{MR3349311,benzi2002preconditioning,ciaramella2022iterative,MR1990645}
and the references therein.
There is no matrix-matrix multiplication in preconditioning algorithms,
such as PCG,
as they do not form the products explicitly; and only matrix-vector
multiplications/divisions are performed. For example, the well known
\textdef{Jacobi preconditioner} uses the diagonal $d=\diag(A)$, if it is
not zero, and the diagonal matrices $P_2=\Diag(d), P_1=I$.

A discussion and references on the two condition numbers $\kappa,\omega$
is given recently in \cite{JungTorrWolkOmega2024}.
The matrix nearness problem $\min_P \| I - AP\|_F$ is often  used to
find preconditioners $P$. This also arises in the derivation of
quasi-Newton methods.
The $\omega$-condition number is introduced in \cite{DeWo:90} as a measure for
nearness to the identity using $\min_{P\succ 0} \omega(MP)$
for finding optimal
quasi-Newton updates; see also~\cite{qz,WoZh:93,DoanW:11}. 
It is related to the measure $\trace A -\log \det A$ used in the 
convergence proofs in~\cite{ByrNoc:89,ByrNocYua:87}.
Then,  Kaporin \cite{MR1277801} used $\omega$  to derive new conjugate
gradient convergence rate estimates and guarantees. More recently
\cite{bock2025connectingkaporinsconditionnumber} 
presents further relationships and convergence analyses.
As mentioned above, it is known that
the ratio of AM to GM is approximately the 
\textdef{coefficient of variation, \CV}, 
$AM/GM \approx 1+\frac 12 CV^2$ \cite{d687d7d6-f7f0-360e-8186-5a92edcffeb7}. 
When the data is log-normal, then we get the exact relation $AM/GM=\sqrt{1+CV^2}$.
Thus by reducing CV,
$\omega$ promotes clustering of eigenvalues/singular values which
is essential in convergence in PCG type methods,
e.g.,~\cite{knyazev2001toward,MR98j:65023}.

Note that for invertible $A\in \Mn$, \cite{doi:10.1287/opre.2022.0592} uses
the former in $\sqrt{\kappa(A^TA)}=\kappa(A)$ as both involve largest 
to smallest singular values. In fact, as emphasized 
in~\cite{JungTorrWolkOmega2024}, $\kappa(A) = \kappa(A^{-1})$ and it is
the latter that is the operation for solving linear systems.
However, these statements are not true for $\omega$, as the
numerator being the sum of singular values, the nuclear norm, is not
equivalent to $\sqrt{\trace(A^TA)}$ and $\omega$ is not inverse
invariant. In this paper,
for $\omega$, we work with general invertible $A$ and then use
$\omega(A^TA)$, i.e.,~use the ratios of eigenvalues of $A^TA$.

In contrast to using $\kappa$, for preconditioners with special
structure one can exploit
the simple derivatives of $\trace,\det$ in  $\omega$ and obtain explicit
formulae for optimal preconditioners. Moreover, these often coincide with
heuristics in the literature, see~\cite{JungTorrWolkOmega2024}. 
Thus no optimization algorithm is needed to obtain the optimal
preconditioner. For example, if we restrict to a diagonal
preconditioner $D=\diag(d)$, then the classic Jacobi 
preconditioner is a multiple of the $\omega$-optimal diagonal preconditioner. 
If we restrict to a diagonal structure along with a partial upper
triangular part size $k$, 
then the diagonal part again corresponds to the Jacobi
preconditioner, while the upper triangular part 
$D_{ij}, i\leq j\leq k$, yields exactly the
formula for the Cholesky factorization.
(See~\Cref{prop:blckdiaf} for more details).
Both these observations highlight the
close connections $\omega$-optimal preconditioners have 
with classical heuristic preconditioners.

\subsubsection{Notation} \label{sect:notation}
\index{$\Mn$, matrices order $n$}
\index{$\Sn$, symmetric matrices order $n$}
\index{$\Snp$, cone of positive semidefinite matrices}
\index{$\Snpp$, cone of positive definite matrices}
\index{matrices order $n$, $\Mn$}
\index{symmetric matrices order $n$$, \Sn$}
\index{cone of positive semidefinite matrices, $\Snp$}
\index{cone of positive definite matrices, $\Snpp$}
We work in \emph{real} Euclidean vector spaces.
We let $\Sn$ denote the space of symmetric matrices with the trace
inner product and corresponding Frobenius norm; $\Snp,\Snpp$ are the
cones of positive semidefinite, and definite, symmetric matrices of order $n$,
respectively; denoted $S\succeq 0, S\succ 0$, respectively.
We let $\Mn$ denote the
space of square matrices of order $n$, also equipped with the trace inner
product and Frobenius norm. We only work with real matrices in  this
paper. Throughout, we let $A \in \Mn$ be a nonsingular matrix. 
Hence, \textdef{$M=A^TA\succ 0$}.

\index{$\kappa(A)$, classical condition number}
In the literature, for $A\in\Mn$, the \textdef{classical condition number,
$\kappa(A)$}$ = \|A\|\|A^{-1}\| = \frac {\max \sigma_i(A)}{\min
\sigma_i(A)}$, i.e.,~the ratio of largest to smallest singular values.
For $B\in \Mn$ with real eigenvalues we let
\[
\lambda_i = \lambda_i(B): \quad 
\textdef{$\lambda_{\max}=\lambda_1$}\geq \lambda_2\geq \ldots \geq
\textdef{$\lambda_{\min} = \lambda_n$},
\]
be the eigenvalues in nonincreasing order, and thus define:
\textdef{$\lambda_1(B) = \max_i\lambda_i(B)$};
\textdef{$\lambda_n(B) = \min_i\lambda_i(B)$}. 
We let $I_{n}\in \Mn$ be the identity matrix and let
$e_n\in \Rn$ be the vector of ones. 
We often use
the following matrix for a basis for the orthogonal complement
$e_n^\perp$,
\begin{equation}\label{eq:defV}
\textdef{$V =
\frac 1{\sqrt 2}\begin{bmatrix} I_{n-1} \cr -e_{n-1}^T \end{bmatrix}$}, \quad
\range(V) = e_{n}^\perp.
\end{equation}
We use $X\bullet Y$ to denote the \emph{Hadamard
product} (elementwise product) of two (compatible)
matrices. For a vector $d\in \Rn$,
\textdef{$\Diag(d)\in \Sn$} denotes the diagonal matrix formed using $d$.
The adjoint linear transformation for $S\in \Sn$ is denoted
$\Diag^*(S)=\,$\textdef{$\diag(S)\in \Rn$}.
For two compatible functions $f$ and $g$, we let $f\circ g$ 
denote the composite function.  
\index{orthogonal complement, $e_n^\perp$}
\index{$e_n^\perp$, orthogonal complement}
\index{$X\bullet Y$, Hadamard product}
\index{Hadamard product, $X\bullet Y$} 
\index{$f\circ g$, composite function}
\index{composite function, $f\circ g$}

The
classical Fenchel subdifferential of a convex function $h$
is defined as
\begin{equation}\label{Fenchel subdifferential}
\textdef{$\partial h(x):= \left\{v:\langle v, y-x\rangle \leq h(y)-h(x), \quad \forall y\right\}$}.
\end{equation}
The Clarke subdifferential of a locally Lipschitz,
not necessarily convex, function $h$ is defined as 
\begin{equation}\label{Clarke subdifferential}
\partial_{C}h(x)=\conv\left\{s:\exists x^k \rightarrow x, \nabla h(x^k)
\textrm{ exists, and } \nabla h(x^k)\rightarrow s \right\},
\end{equation}
where $\conv$ denotes the convex hull of a set. It is well known that the Fenchel subdifferential
and the Clarke subdifferential coincide for convex functions. Further notation is introduced below as needed.
\index{$\conv$, convex hull}
\index{convex hull, $\conv$}

We now recall the well-known facts about the largest and smallest
eigenvalues of symmetric matrices and include a proof for completeness
as it is useful further below.
\begin{lemma}
\label{lem:lam1nconv}
The maximum and minimum eigenvalue functions on $\Sn$,
$\lambda_{\max},\lambda_{\min}: \Sn \to \R$,
are convex, concave, respectively.
\end{lemma}
\begin{proof}
Letting $B \in \Sn$ and using the Rayleigh quotient, we get
$\lambda_{\max}(B) = \max \{ x^T Bx : \|x\|=1\}$,  which is a maximum of linear (so convex) functions in $B$,  and therefore
is convex. We get the concavity part similarly from
$\lambda_{\min}(B) = \min \{ x^T Bx : \|x\|=1\}$.
\end{proof}

Now let 
\[
M\in  \Snpp, d\in \Rpp^n, D=\Diag(d). 
\]
We note the following useful relations that follow by the commutativity 
property for eigenvalues, the fact that similarity transforms preserve the spectrum, and our definitions of $\kappa, \omega$ that
use eigenvalues for products of positive definite matrices:

\section{$\kappa$-Optimal Diagonal Preconditioning}
\phantomsection
\label{sect:otptheory}
\index{$\kappa$-optimal diagonal preconditioning}

Preconditioning (or scaling) plays a fundamental role in numerical
linear algebra and optimization; see, for example, the
surveys~\cite{MR3349311,benzi2002preconditioning,ciaramella2022iterative}). 
In this section, we study the problem of computing a {\bf
$\kappa$-optimal diagonal preconditioner} for positive definite linear
systems
\[
Mx=b,\quad M=A^TA,\quad  A\in \Mn \text{  nonsingular}.
\]
Our objective is to find a diagonal matrix $D\succ 0$ that minimizes the
classical condition number
$\kappa(D^{1/2} M D^{1/2})$.\footnote{Note that we are using 
$\kappa(D^{1/2}MD^{1/2})$ and not the inverse 
$D^{-1/2}$ or $D^{-1}$ as is customary, see e.g.,~\cref{eq:precondPinvAPinv}. 
This simplifies the derivatives in the optimization problems.
Moreover, the equivalent expression for $\kappa(D^{1/2}MD^{1/2})$ 
follows from the
invariance of the spectrum after commuting matrices.}

Rather than viewing this as a nonlinear problem in $D^{1/2}$, we exploit
the eigenvalue invarience in 
\Cref{eq:kappaomegadefequivs} to reformulate to the 
simpler affine mapping $MD$ and still minimize
$\kappa(D^{1/2}MD^{1/2}$). The reformulation yields a pseudoconvex
minimization problem with several advantages:
\begin{itemize}
\item
all stationary points are global minimizers,
\item
subgradients are easy to compute,
\item
the composite functions are used to avoid
homogeneity and design very efficient and scalable algorithms.
\end{itemize}

Note that \cite{gao2023scalable} also considers minimizing $\kappa(D^{1/2}MD^{1/2})$, although they aim to find a diagonal preconditioner that is a convex combination of a small basis or list of diagonal preconditioners. 

\subsection{Optimality Conditions}
\label{sect:optkappadiag}
Let $M\in\Snpp$. We denote the quadratic type scaling as follows:
\[
\cD: \Rnpp \to \Sn, \quad
\textdef{$\cD(d) = \Diag(d)^{1/2}M\Diag(d)^{1/2}$}.
\]
By abuse of notation, we let the argument determine the function,
\[
\textdef{$\kappa(d) = \kappa(\cD(d)) = (\kappa\circ \cD)(d)$}.
\]
Similarly,
\[
\textdef{$\lambda_{\max}(d) = \lambda_{\max}(\cD(d)) =
(\lambda_{\max}\circ \cD)(d)$}, \quad
\textdef{$\lambda_{\min}(d) = \lambda_{\min}(\cD(d)) = (\lambda_{\min}\circ \cD)(d)$}.
\]
In the literature, e.g.,~\cite{gao2024gradientmethodsonlinescaling,gao2023scalable,doi:10.1287/opre.2022.0592},
$\kappa$-optimal diagonal preconditioning refers to solving
\begin{equation}
\label{eq:optkappaprob}
d^* \in \argmin \left\{ \kappa(\cD(d))=\frac {\lambda_{\max}(\cD(d))} {\lambda_{\min}(\cD(d))} :
        d\in \Rnpp \right\}.
\end{equation}
Here we restrict to $d\in \Rnpp$ as we are taking square roots, see also
\Cref{rem:Homogeneity}, below.

We show below, that for our purposes, we can use the
equivalent \emph{linear in $d$} transformation
\begin{equation}\label{linear transform}
\cDo : \Rnpp \to \Sn, \,\, \textdef{$\cDo(d) = M\Diag(d)$}.
\end{equation}

With \textdef{$D = \Diag(d)$},
we first note the relationships in the eigenpairs of 
$\cD(d)=D^{1/2}MD^{1/2}$ and the nonsymmetric but similar $\cDo(d)=MD$, and
$\cDo(d)^T=DM$; as well we have the convexity and concavity of the maximum and
minimum eigenvalues of these nonsymmetric matrices $MD$.\footnote{These
are a special case of so-called $K$-pd matrices in the literature,
i.e.,~a product of two symmetric matrices where at least one is positive
definite. The product of two positive definite matrices $P=AB$ arises in
the optimality conditions in semidefinite programming. One of the
difficulties is that the symmetrization of $P$ is \emph{not} necessarily
positive definite, see e.g.,~\cite{ScTuWolk:03}.
}

\begin{lemma}
\label{lem:eigsADconvconc}
Let $M, D=\Diag(d)\in \Snpp$, and $(u_i,\lambda_i),  i=1,\ldots ,n$, be
orthogonal eigenpairs of $\cD(d)$, i.e.,~$\langle u_i,u_j\rangle = 0, \forall
i\neq j$. Define 
\[
x_i=D^{-1/2}u_i, \forall i,\,
 X=\begin{bmatrix}x_1&\ldots&x_n\end{bmatrix}, 
Y= X^{-T}, \, \Lambda =
\Diag(\lambda).
\]
Then $X$ and $Y$ are matrices of right and left eigenvectors of $\cDo(d)$, respectively, with corresponding matrix of eigenvalues $\Lambda$; and $Y$ is given by $Y=DX (X^TDX)^{-1}$.
\\Moreover, $\lambda_{\max}(d)=\lambda_1(d)$ 
(respectively, $\lambda_{\min}(d)=\lambda_n(d)$) is
a convex (respectively, concave) function of $d\in \Rnpp$.
\end{lemma}
\begin{proof}
Let $U:=[u_1 \, \ldots \, u_n]$ be the corresponding matrix with
\emph{orthogonal}, not necessarily orthonormal,
eigenvectors. We have $X=D^{-1/2}U$ and thus $U^TU = X^TDX$. This brings us to
\[
U^{-T} = U(X^TDX)^{-1}.
\]
Notice that $X^TDX$ is a diagonal matrix with its diagonal entries given by $\|u_i\|^2$, which are strictly positive as eigenvectors are nonzero. Thus the inverse is well-defined.
Now, observe that
\[
\begin{array}{rcl}
	Y &=& X^{-T} = D^{1/2}U^{-T}\cr
	&=& D^{1/2}U(X^TDX)^{-1}\cr
	&=& DX (X^TDX)^{-1}
\end{array}
\]
as desired. Now,
\[
\begin{array}{rcl}
D^{1/2} M D^{1/2} U = U\Lambda 
&\implies&
  M D D^{-1/2} U = D^{-1/2} U \Lambda  
\\&\implies&
MDX = X \Lambda
\\&\implies&
X^TDM = \Lambda X^T
\\&\implies&
DMX^{-T} = X^{-T}\Lambda
\\&\implies&
DMY =  Y\Lambda.
\end{array}
\]
The second and last equation show that $X$ and $Y$ are right, and left, eigenvectors of $MD$, respectively.
Moreover, we note that $M\succ 0$ has a positive definite square root
and  $\lambda_i(M\Diag(d)) = \lambda_i(M^{1/2}\Diag(d)M^{1/2})$. Therefore,
\[
\max \lambda_i( \cDo(d)) = 
  \max \lambda_i(M^{1/2}\Diag(d)M^{1/2}) = 
  \max_{\|x\|=1} x^TM^{1/2}\Diag(d)M^{1/2}x.
\]
As in the proof of~\Cref{lem:lam1nconv},
the latter is a maximum of linear functions in $d$ and therefore is
convex. The concavity result follows similarly.
\end{proof}

\begin{cor} \label{cor:kappaext}
Let $M\in \Snpp, d\in  \Rnpp$. Then
\[
\textdef{$\lambda_{\max}(d) = \lambda_{\max}(\cD(d)) =
\lambda_{\max}(\cDo(d))$},\quad
\textdef{$\lambda_{\min}(d) = \lambda_{\min}(\cD(d)) = \lambda_{\min}(\cDo(d))$};
\]
and
\[
\textdef{$\kappa(d) =
\frac{\lambda_{\max}(\cD(d))}{\lambda_{\min}(\cD(d))}=\frac{\lambda_{\max}(\cDo(d))}{\lambda_{\min}(\cDo(d))}$},
\]
where recall $\cDo(d)$ is as in \Cref{linear transform}.
\end{cor}
\begin{proof}
The proof follows immediately from \Cref{lem:eigsADconvconc}.
\end{proof}
This allows us to consider the equivalent simplified view
of finding a $\kappa$-optimal diagonal
preconditioner, i.e.,~we need to solve the fractional pseudoconvex\footnote{A function
$f:X\subseteq \mathbb R^{n}\rightarrow \mathbb R$ is said to be pseudoconvex 
if for 
all $x,y\in X$,
$\nabla f(x)\cdot(y-x)\geq 0 \implies f(y)\geq f(x)$.
A key property of pseudoconvex functions
is that every stationary point is a global minimizer.
}
minimization problem
\begin{equation}
\label{eq:probdDoforkappa}
\bar d \in \argmin \left\{\kappa(d)= \frac {\lambda_{\max}(\cDo(d))}
{\lambda_{\min}(\cDo(d))} :
        d\in \Rnpp \right\}.
\end{equation}

The following \Cref{lem:KposdefInertia} shows that for $M\succ 0$, 
the matrix $MD$ having all positive eigenvalues
is equivalent to the positive definiteness of
$D$. Combining
this observation
with  
\Cref{lem:eigsADconvconc}, we obtain 
that 
$\kappa(d)$ is the ratio of a convex function and a positive concave function on its domain, and is therefore pseudoconvex.

\begin{lemma}
\label{lem:KposdefInertia}
Let $M\in \Snpp,  d\in \Rn,  D=\Diag(d)$. Then 
\[
\lambda_i(MD)>0, \forall i\,\,  \iff \,\, d\in \Rnpp.
\]
\end{lemma}
\begin{proof}
We note that the eigenvalues of $MD$ are the same as the
eigenvalues of $M^{1/2}DM^{1/2}$. Sylvester's Lemma of inertia 
implies that the number of negative eigenvalues of $M^{1/2}DM^{1/2}$ is
the same as that of $D$ and so it is the same as the number of negative elements in $d$.
\end{proof}

\begin{remark}
\label{rem:Homogeneity}
Note that the two equivalent problems~\Cref{eq:optkappaprob} 
and~\Cref{eq:probdDoforkappa} are both essentially \underline{unconstrained}, 
and the optima are attained and characterized as
stationary points in $\Rnpp$. This is not obvious but follows since we have
the ratios of convex and concave functions using $\lambda_{\max},\lambda_{\min}$
and this is over the open cone constraint $d\in\Rnpp$. If we add the
constraints $d^Te_n=n, d>0$, or equivalently that $d = e_n+Vv > 0$, with $V\in\R^{n\times(n-1)}$ a matrix whose columns contain a basis for $e_n^\perp$ as done above, 
then we have a bounded problem in $v\in\R^{n-1}$ and the spectral
radius $\lambda_{\max}(MD) \leq \|M\|\|D\|$ is
bounded above. Bounded below away from $0$ follows from applying the
greedy solution to the knapsack problem with
 constraints $\sum_i d_i = n, d\geq 0$. We get
 \[
 \begin{array}{rcl}
 \lambda_{\max}(d) 
 &\geq & 
 \trace(MD)/n =\sum_i M_{ii}d_i/n \geq 
 \min_i M_{ii}>0.
 \end{array}
 \]
Therefore,
with the denominator going to $0$, we have $\kappa$ going to $\infty$
as $d = e_n+Vv$ approaches the boundary of the simplex.
That is, minimizing $\kappa$ provides a \emph{self-barrier function}  for this
boundary.

Moreover, $\alpha > 0 \implies \kappa(d)=\kappa(\alpha d)$, i.e.,~we
have positive homogeneity of degree zero. Therefore, the opimal $d$ is \emph{not}
unique. By pseudoconvexity, the optimal set is a convex set. 
This set can be large, see~\Cref{prop:nonuniqd}.
\end{remark}

\subsubsection{Characterization of $\kappa$-Optimal Scaling}
Using eigenvalue derivative formulas in
\Cref{lem:eigderiv,lem:optkappacond}, we compute the gradient of
$\kappa(d)$ in the smooth setting. When the largest and smallest
eigenvalues of $\cDo(d)$ are simple, the gradient takes an explicit form
in terms of the corresponding eigenvectors. The
main result \Cref{thm:optcondkappa} provides a characterization for
$\kappa$-optimality.

\begin{lemma}[{\cite[(3.1)]{MR2349925}}]
\label{lem:eigderiv}
Let $B:\R\to \Mn$ be differentiable with derivative $\dot B=\dot B(t)$,
and, at $t=\bar t$, let $B(\bar t)$  be 
diagonalizable with real eigenpairs (ignoring the argument $\bar t$)
\[
BX=X\Lambda, \, B^TY=Y\Lambda, \quad \Lambda = \Diag(\lambda).
\]
And make the choice $Y = X^{-T}$.
Let $1\leq k\leq n$ and $\lambda_k$ be a singleton eigenvalue with
right and left eigenvector $x_k,y_k$, respectively, taken from the
corresponding columns of $X,Y$, respectively. Then the derivative of eigenvalues is given by 
\begin{equation}
\label{eq:derivlamk}
\dot \lambda_k(\bar t) = y_k^T(\bar t) \dot B(\bar t) x_k(\bar t) =
\trace \dot B x_ky_k^T.\footnote{If $Y$ is an invertible matrix of
right eigenvectors not chosen using $X^{-T}$, then the normalization
scaling needs to be added explicitly as 
$\dot \lambda_k(\bar t) = y_k^T(\bar t) \dot B(\bar t) x_k(\bar
t)/(y_k(\bar t)^Tx(\bar t))$.
}
\end{equation}
\end{lemma}
\begin{proof}
The proof is in~\cite[Pg 303]{MR2349925}.\footnote{The derivative of
eigenvectors is included in the reference along with a useful
normalization that allows for stability of the evaluations of the
eigenvectors.}
\end{proof}

\begin{lemma}
\label{lem:optkappacond}
The Fr\'echet derivative of the linear transformation $\cDo$ at $d$
acting on $\Delta d$ is (simply)
\[
\textdef{$\cDo^\prime(d)(\Delta d)$} = M\Diag(\Delta d).
\]
Now, let $\lambda$ be a singleton eigenvalue of
$\cDo(d)$ with a right eigenvector $x$.
Then the gradient of the composite function at $d$ is
\[
\nabla \lambda(  \cDo(d)) = \frac{\lambda}{x^TDx}x\bullet x.
\]
\end{lemma}

\begin{proof}
The derivative of the linear transformation is clear.

Suppose as above that we have a linear transformation $\cDo(d)$ with
singleton eigenvalue $\lambda$. Let $y = Dx/x^TDx$ be the corresponding
left eigenvector given in \Cref{lem:eigsADconvconc}. Then the derivative of the composite
function at $d$ acting on $\Delta d$ is
\[
\begin{array}{rcl}
\langle \nabla(\lambda\circ\cDo)(d), \Delta d\rangle
&=&
\lambda^\prime \left( \cDo^\prime(d)(\Delta d)\right)= y^T\left(\cDo'(d)(\Delta d)\right)x
\\ &=&
 \frac{1}{x^TDx}x^T\left(DM\Diag(\Delta d)\right)x
\\ &=& \frac{\lambda}{x^TDx}x^T\Diag(\Delta d)x
=\frac{\lambda}{x^TDx}\langle x\bullet x,\Delta d\rangle.
\end{array}
\]
\end{proof}

\Cref{thm:optcondkappa} shows that
at optimal preconditioning, the extreme eigenvectors must ``spread mass
equally'' across coordinates. This balance condition is what
characterizes optimal diagonal preconditioners.
\begin{theorem}[Characterization of $\kappa$-optimal diagonal
preconditioning]
\label{thm:optcondkappa}
Consider $M\in\Snpp, d\in \Rnpp$ be given. Let $D=\Diag(d)$ and
let $(\lambda_1,x_1)$ (resp. $(\lambda_n,x_n)$) be
a singleton maximum (resp. minimum) eigenpair of 
$\cDo(d)$.
Then the gradient of the composite function $\kappa(d)$ is
\[
\nabla \kappa(d) = \kappa (d)\left(
\frac 1{x_1^TDx_1} (x_1\bullet x_1)-\frac 1{x_n^TDx_n}(x_n\bullet x_n)
\right),
\]
where recall that $\kappa(d)$ is as in \Cref{cor:kappaext} .
Hence, $M$ is $\kappa$-optimally diagonally preconditioned
if, and only if, the orthonormal eigenvector pair satisfies
\begin{equation}
\label{eq:optkappascaled}
x_1\bullet x_1 = x_n\bullet x_n.
\end{equation}
Equivalently, after a permutation of the elements to account for the
sign,
\begin{equation}
\label{eq:optkappascaledsigns}
x_1 = \begin{pmatrix} u \cr v \end{pmatrix},\,
x_n = \begin{pmatrix} u \cr -v \end{pmatrix},\,
\|u\| = \|v\|.
\end{equation}
In the nonsmooth case, we can choose the normalized
$x_1$, respectively $x_n$, in the eigenspace of 
$\lambda_{\max}(\cDo (d))$, respectively $\lambda_{\min}( \cDo (d))$.

\end{theorem}
\begin{proof}
From \Cref{lem:optkappacond}, we can find the gradient as follows:
\[
\begin{array}{rcl}
\nabla \kappa(d) 
&=& 
\frac 1{\lambda_n^2} \left(
\lambda_n \dot\lambda_1 -  
            \lambda_1\dot\lambda_n\right)=\frac 1{\lambda_n^2} \left(\frac{\lambda_n\lambda_1}{x_1^TDx_1}x_1\bullet x_1 - \frac{\lambda_1\lambda_n}{x_n^TDx_n}x_n\bullet x_n\right)
\\&=&
\kappa(d) \left(
(x_1\bullet x_1)/(x_1^TDx_1) -  (x_n\bullet x_n)/(x_n^TDx_n)\right).
\end{array}
\]
The characterization for $\kappa$-optimally diagonally preconditioned matrix $A$ follows from solving $\nabla \kappa(e_n)=0$.
\end{proof}

Because $\kappa(d)$ is positively homogeneous of degree zero, the
minimizer is not unique. \Cref{prop:nonuniqd} shows that this nonuniqueness
is structural rather than pathological, i.e.,~when the eigenspace
for the largest and smallest eigenvalues is unique of dimension $2$, the
orthogonal complement is unique.
When the extreme eigenvectors do
not uniquely determine the eigenspace of the
orthogonal complement, entire continua of
optimal diagonal scalings exist. This observation motivates the
reformulation in \Cref{Avoiding Pos Homogeneity}, where homogeneity is removed.

\begin{prop}[Nonuniqueness of $\kappa$-optimal diagonal preconditioning]
\label{prop:nonuniqd}
Let $M\succ 0$ be a $\kappa$-optimal 
diagonal preconditioned matrix as given in~\Cref{thm:optcondkappa} with
$x_i,\lambda_i, i=1,n$, being two, \underline{singleton}, 
eigenpairs that satisfy the 
optimality conditions in \Cref{thm:optcondkappa}. 
Thus we have 
\[
\lambda_1 > \lambda_2\geq \ldots\geq \lambda_{n-1} > \lambda_n>0,
\quad \lambda = (\lambda_i)\in \Rnpp,\, \Lambda = \Diag(\lambda),
\]
and we let $Q = \begin{bmatrix} x_1 & \bar Q & x_n\end{bmatrix}$ be an
orthogonal matrix and $M=Q\Lambda Q^T$ from the spectral theorem. Let $V$ be the basis for $e_n^\perp$ defined in~\Cref{eq:defV}.
Suppose in addition that
the \emph{lack} of \textdef{strict complementarity condition} holds,
i.e.,~there exists $v$ such that
\begin{equation}
\label{eq:nonuniqcond}
0\neq v\in \{u\in \Rnm : 0 = Vu\bullet x_1 = Vu\bullet x_n\}.
\end{equation}
Then with $\Delta D = \Diag(Vv)$, there exists $\epsilon >0$ such that
\[
\kappa(M) = \kappa\left((I+t\Delta D)M(I+t \Delta D)\right), \,
\forall |t|\leq \epsilon.
\]

\end{prop}
\begin{proof}
The result follows from expanding
\[
(I+\epsilon \Delta D)M(I+\epsilon \Delta D) = M + O(\epsilon)
\]
and using the continuity of eigenvalues and the fact that the optimality
conditions imply
\[
\begin{array}{rcl}
	Vv\bullet x_i = 0, i=1,n &\iff & \Diag(Vv)x_i = 0, i=1,n \\
	&\iff & \Delta Dx_i = 0, i=1,n.
\end{array}
\]
Specifically, let 
\[
M_2=\bar Q\Diag((\lambda_2,\ldots,\lambda_{n-1})^T)\bar Q^T, \quad  M_1=M-M_2.
\]
The above equation then implies $\Delta DM_1 = M_1\Delta D = 0$.
Moreover, let $D=(I+\epsilon\Delta D)$, then
\[
DMD = M_1 + DM_2D = M + \epsilon\Delta DM_2 + \epsilon M_2\Delta D + \epsilon^2\Delta DM_2\Delta D.
\]
Since $\range(\Delta D)\subseteq\nul(M_1) = \range(M_2)$,
the range of the perturbation of $M$ above is restricted to 
the eigenspace of $M_2$, i.e.,~to the $\spanl(\{x_2,x_3,\dots,x_{n-1}\})$.
This means that after the perturbation, 
with $\epsilon>0$ sufficiently small, $\lambda_1$ and $\lambda_n$ remain
 the largest and smallest eigenvalues of $DMD$ respectively.
As stated above, we are using the continuity of eigenvalues
and the orthogonality $M_1M_2=0$ that arises using the spectral theorem
and $M_1,M_2\succeq 0$.
\end{proof}
Note that if condition
\cref{eq:nonuniqcond} holds, then we get a nonsingleton set 
in $\Rnm$ of solutions. Moreover, the structure of $V$ and support of
$v$ implies that 
\cref{eq:nonuniqcond} restricts the support of the eigenspace
$\spanl(\{x_1,x_n\})$ and we can permute to get the support in the last
entries.\footnote{Necessity of \cref{eq:nonuniqcond} for
nonuniqueness is still an open problem.}

\subsection{A Projected Subgradient Method for Minimizing $\kappa(d)$}
\label{Asymptotic Section}
In \Cref{thm:optcondkappa}, we derived the gradient and optimality conditions for minimizing the pseudoconvex function $\kappa(d)$ in \Cref{eq:probdDoforkappa}, over the open set $d>0$.
Convergence of subgradient methods for pseudoconvex minimization typically requires optimizing over a closed set. Hence, we consider the following optimization problem:

\begin{equation}
\label{eq:DClosedSet}
\bar d \in \argmin \left\{\kappa(d)= \frac {\lambda_{\max}(\cDo(d))}
{\lambda_{\min}(\cDo(d))} :
   d\in \textdef{$\Omega:=\left\{d: d\geq \delta e_n\right\}$}
      \right\},  \quad 0<\delta\ll 1.
\end{equation}

We now exploit the pseudoconvex structure that guarantees
all stationary points are global minimizers making subgradient methods natural.
The search directions used in \Cref{GD} are shown to lie in the
quasisubdifferential of $\kappa$, allowing us to directly invoke
classical convergence results for quasiconvex minimization.
We first treat the unconstrained formulation in \Cref{GD} and
address homogeneity in \Cref{Avoiding Pos Homogeneity}.

\begin{algorithm}[H]
\caption{A Subgradient Method for Minimizing $\kappa(d)$ over $\Omega$}
\label{GD}
\begin{algorithmic}[1]
\Require symmetric positive definite matrix $M \succ 0$; sequence of
positive stepsizes $\{t_k\} \rightarrow 0$ with $\sum_{k=1}^{\infty} t_k=\infty$, scalar $\delta\in (0,1)$; 
tolerance tol; rule for the stopping criterion, stopcrit.
\Statex set $k \gets 0$;
\Statex set $\textrm{stopcrit} \gets \infty$;
\Statex set $d_1=e_n\in \Rn$;
\While{stopcrit$>$tol} (main outer loop)
\State  set $k\leftarrow k+1$;
\State compute min eigenpair $(\lambda_n^{k},x_n^{k})$ and max eigenpair $(\lambda_1^{k},x_1^{k})$ of $M\Diag(d_{k})$;
\State 
compute direction
\begin{equation}\label{direction}
s_k=\frac{\lambda_1^{k}}{\lambda_n^{k}}\left(\frac{1}{\langle x_1^{k}, d_k\bullet x_1^{k} \rangle}\left(x_1^{k}\bullet x_1^{k}\right)- \frac{1}{\langle x_n^{k}, d_k\bullet x_n^{k} \rangle}\left(x_n^{k}\bullet x_n^{k}\right) \right);
\end{equation}
\State perform projected gradient step
\begin{equation}\label{GD update}
d_{k+1}=\max\left\{d_k- t_k\frac{s_k}{\|s_k\|}, \delta e_n\right\};
\end{equation}


\State update stopcrit;

\EndWhile (main outer loop)

\Ensure $\hat D:=\Diag(d_{k+1})$.

\end{algorithmic}
\end{algorithm}

\begin{remark}
For \Cref{GD}, we use a
popular choice for stepsize sequence $\left\{t_k\right\}$ in 
subgradient methods, $t_k=1/k$, where $k$ is the iteration index \cite{kiwiel2001convergence}. 
We note that since \Cref{GD} is a subgradient method, in 
general it is not a descent method. 
Finally, the algorithm
is general in that it does not specify the stopping rule stopcrit.
There are several possible rules that the user
can employ in practice for updating stopcrit in step 6. For example, at every iteration the user can
take \textrm{stopcrit} as $\|s_k\|$ or $|\kappa(d_{k+1})-\kappa(d_k)|$.
\end{remark}

We now briefly describe each of the steps of \Cref{GD}. 
First, step~3
computes a maximal and minimal eigenpair of $M\Diag(d_{k})$ to 
construct the search direction $s_{k}$. It should be noted that the user never has to form the 
matrix $M\Diag(d_{k})$ to compute the eigenpairs.
Instead, the user only has to construct subroutines
which compute $M\Diag(d_{k})*x$ and
$(M\Diag(d_{k}))$\textbackslash$x$ efficiently, where $x\in \Rn$.
Second, it will be shown in \Cref{Asymptotic Analysis}
that the direction $s_k$ computed in step~4 lies in the quasisubdifferential 
(also to be defined in \Cref{Asymptotic Analysis}) 
of $\kappa(d_k)$.
Finally, using this direction, step~5 performs 
the projected gradient update $d_{k+1}=\Pi_{\Omega}(d_k-t_k\frac{s_k}{\|s_k\|})$, 
where \textdef{$\Pi_{\Omega}(x):=\argmin_{y\in \Omega}\|x-y\|$} denotes the projection onto $\Omega$. 

\subsubsection{Asymptotic Convergence Analysis}\label{Asymptotic Analysis}
Our convergence analysis of \Cref{GD} relies mostly on
\cite{kiwiel2001convergence} where efficient subgradient methods for
minimizing quasiconvex functions are presented. Under various assumptions in 
addition to quasiconvexity, the author establishes asymptotic 
convergence of subgradient methods with decaying stepsizes. 
Since our 
function $\kappa(d)$ is pseudoconvex, it is quasiconvex. For a more detailed 
discussion related to the assumptions of Kiwiel \cite{kiwiel2001convergence} and how our set-up
satisfies these assumptions, the reader is 
referred to \Cref{Kiweil assumptions}.
In the remaining part
of this subsection, we show that \Cref{GD} asymptotically converges
by showing that it is an instance of the subgradient framework proposed
by Kiwiel. 

First, we need to define a special subdifferential called the quasisubdifferential for a quasiconvex function~$f$.
Define the strict sublevel set or inner slice of $f$ as:
\index{$\Int \bar D$, interior}
\index{interior, $\Int \bar D$}
\begin{equation}\label{Sublevel set}
\textdef{$\bar S(x):=\left\{y \in \Int\bar D: f(y)<f(x) \right\}$},
\end{equation}
where $\Int \bar D$ denotes the interior of the domain of $f$.
The quasisubdifferential of a quasiconvex function $f$ relative to the above sublevel set is defined as
\begin{equation}\label{Quasisubdiff}
\textdef{${\bar{\partial}}^{\circ} f(x):=\left\{g:\langle g, y-x\rangle <0, \quad \forall y \in \bar S(x)\right\}$}.
\end{equation}
To minimize a quasiconvex function $f$ over a closed convex set $X$, Kiwiel proposes in \cite{kiwiel2001convergence} the following
basic subgradient algorithm:
\begin{align*}
x_{k+1}:=\Pi_{X}(x_k-t_k\hat g_k), \quad \hat g_k:=g_k/\|g_k\|, \quad g_k\in {\bar{\partial}}^{\circ} f(x_k), \quad k=1,2,\ldots, \quad x_1\in X,
\end{align*}
where $t_k>0$ are the stepsizes.

Hence, we need to characterize vectors that lie in the quasisubdifferential of $\kappa(d)$ to show that \Cref{GD} is an instance of 
Kiwiel's subgradient framework.  
The result below presents one characterization of vectors that lie in the quasisubdifferential of fractional programs
like the one in our set-up \Cref{eq:DClosedSet}.

\begin{lemma}\label{lem:RefinedQuasisubdiff}
Suppose $f(x):=a(x)/b(x)$ is well-defined for all $x \in X$ where $a(x)$ is a convex function that is 
positive
on $X$, 
$b(x)$ is a concave function that is positive on $X$. 
If for $x\in X$,
$B:=1/b(x)$	and $A:=a(x)/b^2(x)$, then
\begin{equation*}
	B[\partial a (x)]+A[\partial (-b)(x)]\subseteq {\bar{\partial}}^{\circ} f(x).
\end{equation*}
\end{lemma}

\begin{proof}
Let $x \in X$
and consider $B$ and $A$ as in the assumptions
of the lemma.
It follows from the fact that $a$ and
$-b$ are convex functions, $B$ and $A$ are positive scalars, and Fenchel subdifferential calculus rules that
\[B[\partial a(x)]+A[\partial (-b)(x)]=
\partial(Ba)(x)+\partial(-Ab)(x)
\subseteq 
\partial(Ba-Ab)(x).
\]
Now, let  $g \in B[\partial a(x)]+A[\partial (-b)(x)]\subseteq \partial(Ba-Ab)(x)$, and suppose $y \in \bar S(x)$, i.e., $y$ is in the interior of the domain of $f$ and $f(y)<f(x)$.
It then follows that
\begin{align*}
\langle g, y-x \rangle &\leq B a(y)-A b(y)-B a(x)+Ab(x)\\
&=\frac{a(y)}{b(x)}-\frac{a(x)b(y)}{b^2(x)}<0,
\end{align*}
where the first inequality follows from the definition of Fenchel subdifferential,
the equality follows from the definitions of $A$ and $B$, and the last inequality 
follows from the fact that $a(y)/b(y)<a(x)/b(x)$ and $b(y)$ and $b(x)$ are positive. It then follows 
from the definition of 
quasisubdifferential in \Cref{Quasisubdiff}
that $g \in {\bar{\partial}}^{\circ} f(x)$. 
\end{proof}

\Cref{Cor:Translation} constructs a vector that lies in the 
quasisubdifferential of $\kappa(d)$.
\begin{corollary} 
\label{Cor:Translation}
Let $d\in \Omega$ and let $(\lambda_1,x_1)$ and $(\lambda_n,x_n)$ be a maximal and minimal eigenpair of 
$\cDo(d)$, respectively. Then, it holds that
\begin{equation*}
	\frac{\lambda_1}{\lambda_n}\left( \frac 1{x_1^T(d\bullet x_1)}(x_1\bullet x_1) 
     - \frac 1{x_n^T(d\bullet x_n)}(x_n\bullet x_n) \right) \in {\bar{\partial}}^{\circ} \left(\kappa(d) \right).
\end{equation*}

\end{corollary}

\begin{proof}
It follows from \Cref{lem:eigsADconvconc} that $\lambda_{\max}(\cDo(d))$ and $\lambda_{\min}(\cDo(d))$ are convex and concave
functions of $d$, respectively. Also, it is easy to see that $\lambda_{\max}(\cDo(d))$ and $\lambda_{\min}(\cDo(d))$ are 
positive on $\Omega$.  
Hence, it follows from \Cref{lem:RefinedQuasisubdiff} that
\begin{equation}\label{quasisubdiff inclusion}
\frac{1}{\lambda_{\min}(\cDo(d))}\partial \lambda_{\max}(\cDo(d))+\frac{\lambda_{\max}(\cDo(d))}{\lambda^2_{\min}(\cDo(d))} \partial \left(-\lambda_{\min}(\cDo(d))\right) \subseteq {\bar{\partial}}^{\circ} \left(\kappa(d) \right)
\end{equation}
holds for $d \in \Omega$.
Moreover, it follows from \Cref{lem:optkappacond}, the definition of 
Clarke subdifferential in \Cref{Clarke subdifferential}, and the fact that the Fenchel
and Clarke subdifferentials coincide for convex functions that the
following inclusions hold.
\begin{equation}\label{quasi inclusion 2}
\frac{\lambda_1}{x_1^{T}\left(d\bullet x_1\right)}\left(x_1\bullet
x_1\right)\in \partial \lambda_{\max}(\cDo(d)), \quad
-\frac{\lambda_n}{x_n^{T}\left(d\bullet x_n\right)}\left(x_n\bullet
x_n\right)\in \partial \left(-\lambda_{\min}(\cDo(d))\right),
\end{equation}
where here $(\lambda_1,x_1)$ and $(\lambda_n,x_n)$ are maximal and minimal eigenpairs of $\cDo(d)$, respectively.
The result then follows from \Cref{quasisubdiff inclusion} and
\Cref{quasi inclusion 2}.

\end{proof}

\begin{remark}\label{Quasi Remark}
	It follows from \Cref{Cor:Translation} that the update rule \Cref{GD update} in step~5
	is of the form
	\[d_{k+1}=\Pi_{\Omega}\left(d_k-t_k\frac{s_k}{\|s_k\|} \right), \text{ where } s_k \in {\bar{\partial}}^{\circ} \left(\kappa(d_k) \right).\]
\end{remark}

We are now ready to present the main theorem
that shows \Cref{GD} converges asymptotically.

\begin{theorem}[Asymptotic convergence]
Let $\{d_k\}$ be generated by \Cref{GD}.
Let 
\[
\kappa_{*}:=\min\left\{\kappa(d): d\in \Omega\right\};\quad 
	\kappa_{*}^{k}=\min \{\kappa(d_j), j = 1\ldots,k\}.
\]
 Then 
\[
\underline{\lim}_{k \rightarrow \infty} \kappa(d_k)=\kappa_{*}; 
 \quad \text{and  } \kappa_{*}^{k}\downarrow \kappa_{*}.
\]
\end{theorem}

\begin{proof}
It follows from the last remark in \Cref{lem:eigsADconvconc} and the definitions of $\kappa(d)$
and $\Omega$ in \Cref{eq:DClosedSet}, that $\lambda_{\max}(\cDo(d))$ (resp.
$\lambda_{\min}(\cDo(d))$) is a convex (resp. concave) function that is positive
on $\Omega$. It then follows from this observation, \Cref{Remark Kiwiel Fractional}, and the definition of
$\kappa(d)$, that
assumptions \textbf{A1-A4} in \Cref{Kiweil assumptions} hold for the minimization problem in 
\Cref{eq:DClosedSet}. Clearly, also assumption \textbf{A5} in \Cref{Kiweil assumptions} also holds since $\Omega$ in \Cref{eq:DClosedSet}
is a closed set and the intersection of $\Omega$ with the interior of the domain of $\kappa(d)$ is clearly nonempty. 
The result of the theorem then immediately follows from this observation, 
\Cref{Quasi Remark}, the facts that $t_k\rightarrow 0$, and $\sum_{k=1}^{\infty} t_k=\infty$,
and Theorem 1 in \cite{kiwiel2001convergence}.
\end{proof}

\subsection{Avoiding Positive Homogeneity in Minimizing $\kappa(d)$}
\label{Avoiding Pos Homogeneity}
As mentioned in \Cref{rem:Homogeneity}, $\kappa(d)$ is a
positively homogenous function (of degree zero). Thus there are multiple optimal
solutions and our problem is \emph{Hadamard ill-posed}.
From a computational
stability perspective, it is more efficient to minimize an equivalent 
smaller dimensional formulation that is not positively homogeneous. 
With this in mind,
we let $M\succ 0$ and 
and
$e_n\in
\Rn$ be the vector of all ones.
We consider the function
\begin{equation}\label{non homogoenous function}
\textdef{$\cVo(v):=M\Diag(e_n+Vv)$},
\end{equation}
where $v\in \R^{n-1}$ and $V$ is as in \cref{eq:defV}.
In this subsection, we consider the following formulation
\begin{equation}\label{Not Homogoenous Formulation}
\min \left\{
\textdef{$\kappa(v):=
\frac
{\lambda_{\max}(\cVo(v))}
{\lambda_{\min}(\cVo(v)) }
$} : 
    e_n+Vv\geq \hat \delta e_n,\, v\in \Rnm \right\},
\end{equation}
where 
$\hat \delta\in (0,1)$ is a scalar.
It is easy to see that with the choice of  $V$ in~\cref{eq:defV},
the above \Cref{Not Homogoenous Formulation} can be rewritten as
\begin{equation}\label{Simplex Formulation}
\min \left\{\kappa(v) : v\in \hat \Omega\right\},
\end{equation}
where
\begin{equation}\label{Simplex Def}
\textdef{$\hat \Omega$}:=\left\{v\in \R^{n-1}: \sum_{i=1}^{n-1}v_i\leq -\sqrt{2}(\hat \delta-1), \quad v_i\geq \sqrt{2}(\hat \delta-1), \quad i=1,\ldots n-1 \right\}.
\end{equation}
Note that the set $\hat \Omega$ is convex and compact.
Projecting onto it is also easy since it is just a simplex.
Also, since $\hat \Omega$ is bounded, we will be able to get
nonasymptotic convergence guarantees for the subgradient method 
that we propose to minimize \Cref{Simplex Formulation}.

The following lemma will be useful for developing our subgradient method.
It gives
useful characterizations of the derivatives of
$\cVo(v)$ and $\kappa(v)$ in the smooth setting when the maximum and minimum eigenvalues
of $\cVo(v)$ have multiplicity one.

\begin{lemma}[Derivatives of $\cVo(v),\kappa(v)$]
\label{cor:optandgradcVo}
Let $M\succ 0, v\in \Rnm$ be given and set 
$w=e_n+Vv\in \Rnpp$ and $D = \Diag(w)$. 
Also, let $(\lambda_1,x_1)$ and $(\lambda_{n},x_n)$ be maximal
and minimal eigenpairs of $\cVo(v)$, respectively,
where $\lambda_1$ and $\lambda_n$ have multiplicity one and
$x_1$ and $x_n$ are assumed to be normalized. 
Then the following hold:
\begin{enumerate}
\item
The derivative at v acting on $\Delta v\in \Rnm$ is
\[
 \dot \cVo(v)(\Delta v) := \cVo^\prime(v)(\Delta v) = M\Diag(V\Delta v).
\]
\item
The gradient of the composite function $\kappa(v):=\kappa(\cVo(v))$ is
\begin{equation}\label{gradient of V}
 \nabla \kappa(v) = \kappa(v)V^T
  \left( \frac 1{x_1^T(w\bullet x_1)}(x_1\bullet x_1) 
     - \frac 1{x_n^T(w\bullet x_n)}(x_n\bullet x_n) \right).
\end{equation}
\end{enumerate}
\end{lemma}
\begin{proof}
\begin{enumerate}
\item
The proof follows from the function being affine.
\item
For a singleton eigenvalue $\lambda$ with normalized
right eigenvector $x$, we have the left eigenvector $y=Dx/(x^TDx)$ as in 
\Cref{lem:eigsADconvconc}, which satisfies $x^Ty=1$. Then
\[
\begin{array}{rcl}
\langle \nabla (\lambda\circ \cVo)(v),\Delta v\rangle
&=&
y^T \dot \cVo(v)(\Delta v) x=y^T (M\Diag(V\Delta v) )  x
\\&=&
\langle V^T\diag(\lambda D^{-1}yx^T), \Delta v   \rangle
\\&=&
\frac 1{x^TDx}\langle V^T\diag(\lambda D^{-1}Dxx^T), \Delta v   \rangle
\\&=&
\frac 1{x^T(w\bullet x)}\langle V^T\diag(\lambda xx^T), \Delta v   \rangle
=\frac 1{x^T(w\bullet x)}\langle \lambda V^T(x\bullet x), \Delta v   \rangle.
\end{array}
\]

Therefore the gradient of $\kappa(v):=\kappa(\cVo(v))$ is:
\begin{equation}
\label{eq:evalgradkappa}
\begin{array}{rcl}
\nabla \kappa(v)
&=&
\frac 
   {\lambda_n \dot \lambda_1 -\lambda_1 \dot \lambda_n  } 
{\lambda_n^2}=\frac 1{\lambda_n^2}  
\left(
\frac {\lambda_n}{x_1^T(w\bullet x_1)} \lambda_1 V^T(x_1\bullet x_1)
-
\frac {\lambda_1}{x_n^T(w\bullet x_n)}
\lambda_n V^T(x_n\bullet x_n)
\right)
\\&=&
\kappa(v)V^T
  \left( \frac {1}{x_1^T(w\bullet x_1)}(x_1\bullet x_1) 
         - \frac {1}{x_n^T(w\bullet x_n)}(x_n\bullet x_n) \right).
\end{array}
\end{equation}
\end{enumerate}
\end{proof}

The following remark
shows that, as in \Cref{cor:optandgradcVo}, $\nabla \kappa(v)$ 
lies in the quasisubdifferential of $\kappa(v)$.

\begin{remark}\label{Quasisubdiff Nonhomogeneous}
	Let $M\succ 0$ and $v\in \R^{n-1}$. Also, suppose that $w=e_n+Vv$ and let
	$(x_i, \lambda_i)$, $i=1,n$, be eigenpairs of $\cVo(v)$, where $\|x_i\|=1$. It then follows from
	\Cref{lem:RefinedQuasisubdiff} and a similar argument as in the
	proof of \Cref{Cor:Translation}
	that
	\begin{equation}\label{Subdiff Inclusion}
	\kappa(v)V^T
    \left( \frac 1{x_1^T(w\bullet x_1)}(x_1\bullet x_1) 
     - \frac 1{x_n^T(w\bullet x_n)}(x_n\bullet x_n) \right) \in {\bar{\partial}}^{\circ} \left(\kappa(v)\right)
	\end{equation}
where recall that $\kappa(v):=\kappa(\cVo(v))$.
\end{remark}
We now present our subgradient algorithm for minimizing \Cref{Simplex Formulation}.

\begin{algorithm}[H]
\caption{A Subgradient Method for Minimizing $\kappa(v):=\kappa(\cVo(v))$ over $\hat\Omega$}
\label{Subgrad V}
\begin{algorithmic}[1]
\Require symmetric positive definite matrix $M \succ 0$; $V\in \R^{n \times (n-1)}$ 
a basis matrix for the
orthogonal complement $e^\perp$;
sequence of stepsizes $\{t_k\}=1/\sqrt{k}$; scalar $\hat \delta\in(0,1)$; a tolerance tol$>0$; a rule for the stopping criterion, stopcrit.
\Statex set $k \gets 0$;
\Statex set stopcrit $\gets\infty$;
\Statex set $v_1=0 \in \R^{n-1}$ and $w_1=e_n \in \Rnp$;
\While{stopcrit$>$tol} (main outer loop)
\State  set $k\leftarrow k+1$.
\State compute min eigenpair $(\lambda_n^{k},x_n^{k})$ and max eigenpair $(\lambda_1^{k},x_1^{k})$ of $M\Diag(w_{k})$;
\State 
compute direction
\begin{equation}\label{Dir Nonhm}
g_k=\frac{\lambda_1^{k}}{\lambda_n^{k}}V^{T}\left(\frac{1}{\langle x_1^{k}, w_k\bullet x_1^{k} \rangle}\left(x_1^{k}\bullet x_1^{k}\right)- \frac{1}{\langle x_n^{k}, w_k\bullet x_n^{k} \rangle}\left(x_n^{k}\bullet x_n^{k}\right) \right)
\end{equation}
\State perform projected gradient step
\begin{equation}\label{Descent Nonh}
v_{k+1}=\Pi_{\hat \Omega}\left(v_k-t_k\frac{g_k}{\|g_k\|}\right)
\end{equation}
\hspace{\algorithmicindent}where $\hat \Omega$ is as in \Cref{Simplex Def} and set 
\begin{equation}
	w_{k+1}=e+Vv_{k+1};
\end{equation}


\State update stopcrit;

\EndWhile (main outer loop)

\Ensure $\hat D:=\Diag(w_{k+1})$.

\end{algorithmic}
\end{algorithm}

Several remarks about \Cref{Subgrad V} are now given. First, the matrix 
$V$ does not need to be stored in memory and is actually not needed as input.
All the user needs to input is a subroutine that outputs $Vv$ given a vector $v\in \R^{n-1}$.
Likewise, the matrix $M\Diag(w_{k})$ does not need to be stored.
Second, it follows
from \Cref{Quasisubdiff Nonhomogeneous} that $g_k \in  {\bar{\partial}}^{\circ} \left(\kappa(v_k)\right)$,
which 
is defined in \Cref{Quasisubdiff}. Finally, the projected gradient 
step in \Cref{Descent Nonh} can be performed very efficiently
since projecting onto $\hat \Omega$ just involves computing the root 
of a simple equation.

\subsubsection{Nonasymptotic Convergence  Rate Analysis for Minimizing $\kappa(v)$ on $\hat \Omega$}
In this subsection, we show a nonasymptotic convergence rate for \Cref{Subgrad V} for minimizing
$\kappa(v)$ over $\hat \Omega$. More specifically,
we show that
\[\min_{1\leq k \leq K}\kappa(v_k)-\kappa_{*} \leq \epsilon\]
holds for $K=\mathcal O(1/\epsilon^2)$, where $\kappa_{*}=\min_{v\in \hat \Omega} \kappa(v)$. 
Our nonasymptotic convergence analysis of \Cref{Subgrad V} relies mostly on the results 
from \cite{kiwiel2001convergence}
and \cite{hu2020convergence}. The function $\kappa(v)$ is pseudoconvex
since it is the ratio of a convex function and a positive concave function. Also,
observe that the set $\hat \Omega$ is just a box so it is a convex compact
set that is easy to project onto. The only other assumption that needs
to be verified to be able to show a nonasymptotic convergence rate of \Cref{Subgrad V}
is that the function $\kappa(v)$ is Lipschitz continuous on
$\hat \Omega$. This result is proved in the following proposition.

\begin{prop}\label{Lipschitz Prop}
	The function $\kappa(v)$ is Lipschitz continuous on 
	$\hat \Omega$, i.e.,
    \[
	\|\kappa(v_1)-\kappa(v_2)\|\leq L \|v_1-v_2\|, \quad \forall v_1, v_2 \in \hat \Omega
	\]
where $L>0$ and $\hat \Omega$ is as in \Cref{Simplex Def}.
\end{prop}

\begin{proof}
First, it is easy to see that $\lambda_{\max}(\cVo(v))$ (resp. $\lambda_{\min}(\cVo(v))$) is a convex (resp. concave) function.
It then follows from this observation, the fact $\hat \Omega$ is a compact set that is contained in the domain of both functions,
and Proposition A.48(b) in \cite{lee2003smooth} that $\lambda_{\max}(\cVo(v))$ (resp. $\lambda_{\min}(\cVo(v))$) is $L_1$-Lipschitz (resp. $L_2$-Lipschitz)
on $\hat \Omega$.
Consider now the composite function
$h(v)=g(\lambda_{\min}(\cVo(v)))$ where $g(y)=1/y$.
The above conclusion and the fact
that $g(y)=1/y$
is Lipchitz continuous on any positive open interval
then imply that $h$ is $L_3$-Lipschitz continuous on $\hat \Omega$, where $L_3>0$.


We now show that $\kappa(v)$ is Lipschitz continuous. 
First, it holds that
$|\lambda_{\max}(v)|\leq M_1$ and $|h(v)|\leq M_2$ for $v\in \hat \Omega$
since a Lipschitz function on a compact set is bounded.
Then, for any $v_1 \in \hat \Omega$ and $v_2 \in \hat \Omega$, the following holds:

\begin{align*}
    \|\kappa(\cVo(v_1))-\kappa(\cVo(v_2))\|&=\|\lambda_{\max}(\cVo(v_1))h(v_1)-\lambda_{\max}(\cVo(v_2))h(v_2)\|\\
	&=\|\lambda_{\max}(\cVo(v_1))h(v_1)-\lambda_{\max}(\cVo(v_1))h(v_2)+\lambda_{\max}(\cVo(v_1))h(v_2)-\lambda_{\max}(v_2)h(v_2)\|\\
	&\leq M_1\|h(v_1)-h(v_2)\|+M_2\|\lambda_{\max}(\cVo(v_1))-\lambda_{\max}(\cVo(v_2))\|\\
	&\leq M_1L_3\|v_1-v_2\|+M_2L_1\|v_1-v_2\|=(M_1L_3+M_2L_1)\|v_1-v_2\|,\\
\end{align*}
which immediately implies the statement of the proposition with $L=M_1L_3+M_2L_1$.

\end{proof}

We now state \Cref{Nonasymptotic Theorem}, which displays that \Cref{Subgrad V}
is able to find an $\epsilon$-approximate optimal solution of \Cref{Simplex Formulation}
with a sublinear $\mathcal O(1/\epsilon^2)$ rate of convergence.

\begin{theorem}\label{Nonasymptotic Theorem}
Let $\epsilon>0$ be a given tolerance
and suppose that $K=\mathcal O(1/\epsilon^2)$. It then holds that
\begin{equation}\label{Sublinear Complexity Result}
\min_{1\leq k\leq K} \kappa(v_k)-\kappa_{*} \leq \epsilon,
\end{equation}
where $\kappa_{*}=\min_{v\in \hat \Omega} \kappa(v)$ and $\kappa(v):=\kappa(\cVo(v))$. 
\end{theorem}

\begin{proof}
It is immediate to see that $\hat \Omega$ is a convex compact set and that
$\kappa(v)$ is a quasiconvex function since it is the ratio of a convex
and a positive concave function. It follows from 
\Cref{Lipschitz Prop} that $\kappa(v)$ is Lipschitz continuous on
$\hat \Omega$. Finally, it follows from \Cref{Quasisubdiff Nonhomogeneous} 
that the update \Cref{Descent Nonh} in step 5
of \Cref{Subgrad V} is of the form
$v_{k+1}=\Pi_{\hat \Omega}\left(v_k-t_k\frac{g_k}{\|g_k\|}\right)$
where $g_{k} \in {\bar{\partial}}^{\circ} \left(\kappa(v_k)\right)$. Hence, it follows from these observations, the fact
that $\left\{t_k\right\}=1/\sqrt{k}$, 
and Theorem 3.2(ii) in \cite{hu2020convergence} with $s=1/2$, $\delta=\epsilon$, and $p=1$
that the statement of \Cref{Nonasymptotic Theorem} holds.

\end{proof}

\section{$\omega$-Optimal Structured Preconditioning}
\phantomsection
\label{sect:optomegadiagproc}

This section focuses on obtaining 
$\omega$-optimal preconditioners of special structure for finding
least squares solutions of $Ax=b$, where $A$  might not  be symmetric nor square. In what follows, $A$ will be  assumed to have full column rank, so that the matrix \textdef{$M:=A^{T}A$} is positive definite.  Recall that for $M\succ 0$, we define $\omega(M)$ as
\[
\omega(M) = \frac {\trace(M)/n}{\det(M)^{1/n}},
\]
i.e., it is the ratio of the arithmetic to geometric
means of the eigenvalues of $M$. 
Unlike the classical condition number $\kappa$, the quantity $\omega$ is
smooth, and closely related to eigenvalue clustering.
These properties make it particularly well suited to deriving explicit
solutions to optimality conditions for structured preconditioners.

This section is divided into three main subsections. 
\begin{itemize}[nosep]
\item
\Cref{RightLeftDiagSec} treats one-sided diagonal preconditioning,
distinguishing between right- and left-sided scalings.
\item
\Cref{sect:twosided}
addresses two-sided diagonal scaling, establishing
optimality conditions and linking them to matrix balancing.
\item
\Cref{sect:optblockdiag} extends the analysis to block-diagonal preconditioners,
including triangular block structures.
\end{itemize}

A unifying theme of this section is that many classical
preconditioners—Jacobi scaling, row and column normalization, and
Sinkhorn–Knopp balancing—arise naturally as exact solutions or
stationary points of $\omega$-optimal problems. \Cref{table:omegaoptpreconds}
 summarizes several of these connections.

\begin{table}[ht]
    \caption{Relation of $\omega$-Optimal Preconditioners to the Literature}
    \centering
    \begin{tabular}{|c|c|c|c|}
        \hline
        \textbf{Special Structure} & \textbf{Preconditioner} &
\textbf{Location} & \textbf{Literature} \\
\hline
       Right Diagonal   & Column Normalization  & \Cref{prop:omegaoptDAD}&
    \cite{DeWo:90,DoanW:11} \\
	\hline 
	Diagonal for Symmetric $M\succ 0$  
	     & Jacobi Preconditioner & \Cref{cor:twosidedposdefJacobi}&
    \cite{Jacobi:45}\\
        \hline
        Left Diagonal   & Row Normalization  & \Cref{thm:optleftprecond}&
                            \\
	    \hline
        Two-Sided Diagonal   & Sinkhorn-Knopp/Matrix Balancing & \Cref{thm:twosidedinvert}&
    \cite{soules1991rate,knight2008sinkhorn,scetbon2025gradient}  \\
        \hline
	Right Block-Diagonal &Block QR&
	   \Cref{cor:partialqr}   &\cite{JungTorrWolkOmega2024,DoanW:11} \\
        \hline
	Left Block-Diagonal &Characterization &
	   \Cref{thm:leftblkgeneral}   &  \\
	\hline
    \end{tabular}
    \label{table:omegaoptpreconds}
\end{table}

\subsection{Right- and Left-Sided Diagonal}
\label{RightLeftDiagSec}
Suppose that $\diag(A)$ has no zero entries.  The classical Jacobi preconditioner
\cite{Jacobi:45}, \cite[Sect. 10.2]{saad2019iterative} is defined by:
\begin{equation}
\label{eq:jacobiprec}
P^{-1}A = P^{-1}b, \quad P = \Diag(\diag(A)).
\end{equation}

We note that this can be found by using: a splitting $A=M-N$ so that
$A=\Diag(\diag(A))-N$; or
the following variational problem that 
implicitly finds the best approximation of the identity using diagonal matrices:
\[
\min_d \|A-\Diag(d)\|_F, \quad \Diag(d)^{-1}A \approx I.
\]
In what follows, we reinterpret such classical constructions through the
lens of the $\omega$-condition number.

\subsubsection{Right-Sided  Diagonal}
\label{sect:rightsided}
The following result shows that a $\omega$-optimal
right-sided diagonal scaling of an overdetermined full rank matrix $A$ 
is formed using ($\pm$) column normalization, see \cite{DeWo:90}.
\begin{prop}
\label{prop:omegaoptDAD}
Let $A$ be $m\times n$ full column rank, and
$D := \Diag(\diag(A^TA))^{-1}$.
Then $A D^{1/2}$ is a $\omega$-optimal right-sided diagonal scaling i.e.,
\[
D \in \argmin \left\{\omega(D^{1/2}A^TAD^{1/2}) : 
\diag(D) \in \Rnpp \right\}.
\]
\end{prop}
\begin{proof}
The proof of the result can be found in 
\cite[Proposition 2.1(v)]{DeWo:90} and a modified proof in
\cite[Proposition 2.1(3)]{JungTorrWolkOmega2024}.
\end{proof}

As an immediate specialization, we obtain the following classical result.
\begin{cor}
\phantomsection
\label{cor:twosidedposdefJacobi}
Let $M\in \Snpp$. Then the Jacobi preconditioner, $D=\Diag(\diag(M))^{-1}$,
is the $\omega$-optimal diagonal scaling, i.e.,
\[D\in \argmin \left\{\omega(D^{1/2}MD^{1/2}) : 
 \diag(D) \in \Rnpp \right\}.\]

\end{cor}
\begin{proof}
The proof is immediate from \Cref{prop:omegaoptDAD} with 
$A^{T}A$ replaced by $M\in \Snpp$.
\end{proof}
Thus, in contrast to its heuristic origins, Jacobi scaling is optimal in
the precise sense of minimizing $\omega$.

\subsubsection{Left-Sided Diagonal}
\label{sect:leftsided}
In this section we let $A\in \Mn$ be invertible, and
we consider the optimal left-sided diagonal preconditioning problem:
\begin{equation}
\label{eq:leftsidedoptomega}
\begin{array}{rcll}
& \min &
         \omega\big((\Diag(c)^{1/2}A)^T\Diag(c)^{1/2}A\big)
\\ & \text{s.t.} & c\in\Rnpp.
\end{array}
\end{equation}
For simplicity, let $C:=\Diag(c)$.
We can characterize the $\omega$-optimal left-sided diagonal
preconditioner using the above results in \Cref{sect:rightsided} 
for the right preconditioner.
\begin{theorem}
\label{thm:optleftprecond}
	Let $A\in \Mn$ be nonsingular. Then a $\omega$-optimal
left-sided diagonal preconditioner minimizing
\cref{eq:leftsidedoptomega} is given by $C = \Diag(\diag(AA^T))^{-1}$.
\end{theorem}
\begin{proof}
	Note that
	\[
	\omega\big((\Diag(c)^{1/2}A)^T\Diag(c)^{1/2}A\big) = \omega\big(A^T\Diag(c)A\big) = \omega\big(\Diag(c)^{1/2}AA^{T}\Diag(c)^{1/2}\big).
	\]
	Hence, \Cref{prop:omegaoptDAD} and the above equivalence implies that
	\begin{equation*}
		\begin{array}{rcl}
			C = \Diag(\diag(AA^T))^{-1} &\in& \argmin \left\{\omega(C^{1/2}AA^{T}C^{1/2}) : C = \diag(c) \in \Snpp \right\} \cr
			&=& \argmin \left\{\omega\big((\Diag(c)^{1/2}A)^T\Diag(c)^{1/2}A\big) : C = \diag(c) \in \Snpp \right\}.
		\end{array}
	\end{equation*}
\end{proof}

\subsection{Two-Sided Diagonal}
\label{sect:twosided}
We now consider the problem of finding the two-sided $\omega$-optimal
 diagonal scaling for nonsingular matrices $A\in \Mn$.
This subsection is broken up into two smaller
parts. The first part derives the optimality conditions for
the two-sided $\omega$-optimal scaling problem. 
The second part presents an 
iterative matrix balancing scheme that reduces $\omega$ at every iteration
and whose output is proved 
to be a stationary point of the two-sided $\omega$-optimal scaling problem. 

\subsubsection{Optimality Conditions for Two-Sided Problem}
This part presents the formulation of the two-sided $\omega$-optimal diagonal scaling 
problem and derives its optimality conditions.
Namely, we consider
\begin{equation}
\label{eq:twosidedoptomegaprob}
\begin{array}{rcll}
& \min &
         \omega((\Diag(c)^{1/2}A\Diag(d)^{1/2})^T(\Diag(c)^{1/2}A\Diag(d)^{1/2}))
\\ & \text{s.t.} & c,d\in\Rnpp,
\end{array}
\end{equation}
where $A\in \Mn$ is nonsingular.
For simplicity, let $C:=\Diag(c)$ and $D:=\Diag(d)$.
We also define 
\begin{equation}\label{two sided definitions}
\textdef{$\cQ(c,d) := A^T\Diag(c)A\Diag(d)$},\,\,
\textdef{$f(c,d):= \frac{1}{n}\trace(\cQ(c,d))$},\,\,
\textdef{$g(c,d):= \det(\cQ(c,d))^{1/n}$}.
\end{equation}
Therefore we have the following equivalent problem to \Cref{eq:twosidedoptomegaprob}:
\begin{equation}
\label{equivtwo-sided}
\begin{array}{rcll}
&\min &\textdef{$\omega_{\mathcal Q}(c,d):=\frac{f(c,d)}{g(c,d)}$}
\\ & \text{s.t.} & c,d\in\Rnpp.
\end{array}
\end{equation}
Finally, for notational convenience, we define $\mathrm{inv}:\R^{2n}\to\R^{2n}$ by
\begin{equation}\label{cinvdinv}
\textdef{$\mathrm{inv}(c,d) = \begin{pmatrix}
	\diag\left(\Diag(c)^{-1}\right)\cr
	\diag\left(\Diag(d)^{-1}\right)
	\end{pmatrix}$}.
\end{equation}

The following result provides the gradient of $\omega_{\mathcal Q}(c,d)$.
\begin{prop}\label{Omega Two Sided Gradient}
	The gradient of $\omega_{\mathcal Q}(c,d)$ is given by
	\begin{equation}\label{gradientOmegaTwoSided}
	\nabla \omega_{\mathcal Q}(c,d)=\frac{1}{ng(c,d)} \begin{pmatrix}
\diag(A\Diag(d) A^T) \cr
\diag(A^T\Diag(c) A)
      \end{pmatrix}
- \frac 1n\omega_{\mathcal Q}(c,d)\mathrm{inv}(c,d),
	\end{equation} 
where
$g(c,d)$ is as in \Cref{two sided definitions}, $\omega_{\mathcal Q}(c,d)$
is defined in
\Cref{equivtwo-sided}, and $\mathrm{inv}(c,d)$ is as in \Cref{cinvdinv}.
\end{prop}
\begin{proof}
To compute $\nabla \omega_{\mathcal Q}(c,d)$, we have to compute $\nabla f(c,d)$ and $\nabla g(c,d)$.
It is easy to see from the definition of $f(c,d)$ in \Cref{two sided definitions} that $\nabla f(c,d)$
can be computed as
\begin{equation}\label{eq:gradf}
\nabla f(c,d) = \frac{1}{n}
      \begin{pmatrix}
\diag(A\Diag(d) A^T) \cr
\diag(A^T\Diag(c) A)
      \end{pmatrix}.
\end{equation}
To compute $\nabla g(c,d)$,
observe that it follows from the definition of $g(c,d)$ in \Cref{two sided definitions} and the
formula for the gradient of the determinant that
\begin{equation*}
	\begin{array}{rcl}
		\langle \nabla g(c,d), (\Delta c, \Delta d)\rangle 
&=& 
		\det(A^TA)^{1/n}
\langle \nabla
\det\big(CD\big)^{1/n},(\Delta c,\Delta d)\rangle \\
		&=& \det(A^TA)^{1/n}
          \left\langle
\frac{1}{n}\det\big(CD\big)^{\frac{1}{n}-1}\adj\big(CD
          \big),
       (C\Delta D+\Delta CD)\right\rangle
	\\	&=& 
\frac{1}{n}g(c,d)\left\langle \big(CD\big)^{-1}, 
       (C\Delta D+\Delta CD)\right\rangle
		\\&=& 
\frac{1}{n}g(c,d)
    \trace \left( C^{-1} \Delta C + D^{-1} \Delta D\right).
	\end{array}
\end{equation*}
Hence, it follows from the above equations 
\begin{equation}\label{eq:gradg}
\nabla g(c,d)= 
\frac{1}{n}g(c,d)
\begin{pmatrix}
\diag\left(\Diag(c)^{-1}\right)\cr
\diag\left(\Diag(d)^{-1}\right)
 \end{pmatrix}
=\frac{1}{n}g(c,d)\mathrm{inv}(c,d).
\end{equation}
It then follows from \Cref{eq:gradf}, \Cref{eq:gradg}
and the quotient rule that 
$\nabla \omega_{\mathcal Q}(c,d)$
can be computed as
\begin{align*}
\nabla \omega_{\mathcal Q}(c,d)&=\frac{1}{g(c,d)^2}\left(g(c,d)\nabla f(c,d) - f(c,d)\nabla
g(c,d)\right)
=\frac{1}{g(c,d)} 
\nabla f(c,d) - \frac{\omega_{\mathcal Q}(c,d)}{g(c,d)} \nabla g(c,d)\\
&= \frac{1}{ng(c,d)} \begin{pmatrix}
\diag(A\Diag(d) A^T) \cr
\diag(A^T\Diag(c) A)
      \end{pmatrix}
- 
\frac 1n\omega_{\mathcal Q}(c,d)\mathrm{inv}(c,d).
\end{align*}

\end{proof}

\begin{remark}
It is easy to see that $\nabla \omega_{\mathcal Q}(c,d)$
can be equivalently written in the useful form:
\[\nabla \omega_{\mathcal Q}(c,d) = \frac{1}{n}\begin{bmatrix}
		0 & A\bullet A\\ (A\bullet A)^T & 0	\end{bmatrix}\begin{pmatrix}
			c \\d \end{pmatrix}- 
\frac 1n\omega_{\mathcal Q}(c,d)\mathrm{inv}(c,d).
\]

\end{remark}

We now state our main theorem of this subsection on a
necessary condition for a vector
to be globally optimal for \Cref{eq:twosidedoptomegaprob}.
This relates to optimal matrix balancing discussed below.

\begin{theorem}
\label{thm:twosidedinvert}
Let $A\in \Rnn$ be nonsingular.
Let $c^{*},d^{*}\in \Rnpp$ be
a global optimal solution pair of \Cref{eq:twosidedoptomegaprob}. 
Let $e_{2n} \in \mathbb R^{2n}$ be the vector of all ones, 
$C^{*}=\Diag(c^{*})$, and $D^{*}=\Diag(d^{*})$. Then:
\begin{equation}
\label{eq:neccondtwosided}
      \begin{pmatrix}
\diag(C^{*}AD^{*} A^T) \cr
\diag(D^{*}A^TC^{*} A)
      \end{pmatrix} = \alpha e_{2n}, \, \text{ for some  } \alpha>0.
\end{equation}
Hence, if $c^{*},d^{*}$ 
is a global optimal solution of \Cref{eq:twosidedoptomegaprob},
then all row and column of $\sqrt{C^{*}}A\sqrt{D^{*}}$ have identical
Euclidean norms, i.e.,~optimal two-sided $\omega$-scaling produces
a balanced matrix.
\end{theorem}

\begin{proof}
For $c^{*},d^{*} \in \Rnpp$
to be a global optimal solution of \Cref{eq:twosidedoptomegaprob},
it must satisfy the necessary condition
$\nabla \omega_{\mathcal Q}(c^{*},d^{*}) =0$.
Thus, it follows from \Cref{gradientOmegaTwoSided}
that $c^{*},d^{*}$
must satisfy 
\[
0=\frac{1}{ng(c^{*},d^{*})} \begin{pmatrix}
\diag(AD^{*} A^T) \cr
\diag(A^TC^{*} A)
      \end{pmatrix}
- 
\frac 1n\omega_{\mathcal M}(c^{*},d^{*})
\mathrm{inv}(c^{*},d^{*}).
\]
It follows from taking the Hadamard product with the vector
$\begin{pmatrix}
{c^{*}}\cr
{d^{*}}\cr
 \end{pmatrix}$
on both sides of the above equation as well as
multiplying both sides
by $ng(c^{*},d^{*})$
that the following necessary condition must hold
\begin{equation}\label{necCondforOpt}
      f(c^{*},d^{*})e_{2n}=\begin{pmatrix}
\diag(C^{*}AD^{*} A^T) \cr
\diag(D^{*}A^TC^{*} A)
      \end{pmatrix}.
\end{equation}
Observe that the definition of $f(c,d)$ in \Cref{two sided definitions}
implies that 
if $c^{*},d^{*}$ satisfies the above relation
then any positive scalar multiple 
of $c^{*},d^{*}$ also satisfies
\Cref{necCondforOpt}.
It then follows from this observation and \Cref{necCondforOpt}
that a globally optimal solution $c^{*},d^{*}$ must satisfy the necessary condition
\Cref{eq:neccondtwosided} for some scalar $\alpha>0$.
The last observation of the theorem then immediately follows from the condition in
\Cref{eq:neccondtwosided}.

\end{proof}

\subsubsection{Square-Root Sinkhorn-Knopp Algorithm for Two Sided
Problem}
Let a nonsingular matrix $A\in \Mn$ be given.
We have seen above that the $\omega$-optimal right-sided preconditioner is
equivalent to the column normalization preconditioner.
Similarly, the $\omega$-optimal left-sided preconditioner is
equivalent to the row normalization preconditioner.
Therefore, we can alternate between $\omega$-optimal right- and left-sided
preconditioners and decrease the value of $\omega$ at each iteration.
This yields the following computationally efficient matrix balancing
scheme where we alternate between column and row normalization.
This algorithm can be seen as equivalent to the Square-Root Sinkhorn-Knopp
algorithm and is related to other balancing methods in the literature 
\cite{soules1991rate,knight2008sinkhorn,scetbon2025gradient,MR3231982}.

We prove below in \Cref{thm:cgnce2sided}
that the Square-Root Sinkhorn-Knopp algorithm, namely \Cref{alg:Two-Sided}, converges and 
the outputted two-sided diagonal preconditioner satisfies the necessary condition 
in \Cref{eq:neccondtwosided} for global optimality of a two-sided $\omega$-optimal
preconditioner.
This continues the theme that the $\omega$-measure
provides justification for heuristics used in the literature.

\begin{algorithm}[H]
\caption{Square-Root Sinkhorn-Knopp Algorithm for Two Sided $\omega$-Optimal Preconditioner}
\label{alg:Two-Sided}
\begin{algorithmic}[1]
\Require A square nonsingular matrix $A \in \Mn$, a tolerance tol$>0$, and a rule for the 
stopping criterion, stopcrit. 
\Statex set stopcrit $\gets\infty$;
\Statex set $A_0 \gets A$;
\Statex set $k \gets 0$;
\While{stopcrit$>$tol} (main outer loop)
\State set $k \gets k+1$;
\State 
compute $(\sqrt{d_k})_{i}=1./\|(A_{k-1})_{:,i}\|$, $i=1,\ldots n$, where $(A_{k-1})_{:,i}$ is the $i$-th column of $A_{k-1}$;

\State set $\tilde A_k \gets A_{k-1}\Diag(\sqrt{d_k})$;
\State compute $\sqrt{c_k} \in \mathbb R^{n}$ with $(\sqrt{c_k})_{i}=1./\|(\tilde A_k)_{i,:}\|$, $i=1,\ldots n$, where $(\tilde A_k)_{i,:}$ is the $i$-th row of $\tilde A_k$;
\State update stopcrit;
\State set $A_{k} \gets \Diag(\sqrt{c_k})\tilde A_k$;

\EndWhile (main outer loop)

\Ensure  \[\hat C^{1/2}:=\Diag\left(\prod_{j=1}^{k}
\sqrt{c_j}\right) \text{ and } \hat D^{1/2}:=\Diag\left(\prod_{j=1}^{k}
\sqrt{d_j}\right).\]
\end{algorithmic}
\end{algorithm}
There several viable options for the choice of stopcrit in
\Cref{alg:Two-Sided}. One natural choice is
\[
\text{stopcrit}=\max\left(
\max_i \left| \|(A_k)_{i,:}-1\|\right|,
\max_j \left| \|(A_k)_{:,j}-1\|\right|
\right).
\]
It is shown in \Cref{thm:cgnce2sided} that the sequence $A_k$
converges to a matrix that is balanced, i.e., one that
has column and row norms all equal to one.
The remark below also discusses that
the $\omega$ values of the iterates are monotonically
decreasing.

\begin{remark}
At each iteration, $\tilde A_k$ (resp. $A_k$) is computed using the $\omega$-optimal right-sided (resp. left-sided)
preconditioner. Hence, the $\omega$ values of the iterates are monotonically decreasing,
i.e., $\omega(A_kA_k^{T})\leq \omega(\tilde A_k\tilde A_k^{T})\leq \omega(A_{k-1}x A_{k-1}^{T})$
for all $k\geq 1$. The output of \Cref{alg:Two-Sided} thus satisfies
$\omega((\hat C^{1/2} A\hat D^{1/2})^{T}(\hat C^{1/2} A \hat D^{1/2}))\leq \omega(A^{T}A)$.
\end{remark}

\Cref{thm:cgnce2sided} now provides 
a convergence guarantee for \Cref{alg:Two-Sided} 
and shows that the necessary condition 
\Cref{eq:neccondtwosided} for a two-sided $\omega$-optimal
diagonal preconditioner
holds in the limit for \Cref{alg:Two-Sided}.

\begin{theorem}\label{thm:cgnce2sided}
Let $A$ be such that $A\bullet A$ has total support.\footnote{A nonnegative square matrix $B$ is said to have total support if $B\neq 0$ and all its
nonzero elements lie on a positive diagonal. In this case, a diagonal of a matrix is just a collection of elements with one element from each row and one element from each column of $B$. More specifically, it will consist of $b_{i,\sigma(i)}$ for $i=1,2,\ldots, n$ for a permutation $\sigma$ and where $b_{ij}$ denotes the $(i,j)$ entry of $B$.} Then the sequence $A_k$ converges
linearly to a matrix $\bar A:=\bar C^{1/2}A \bar D^{1/2}$ that has column and row norms all equal to 1.
That is, 
\[\lim_{k\rightarrow \infty} \Diag\left(\prod_{j=1}^{k}
\sqrt{c_j}\right) A \Diag\left(\prod_{j=1}^{k}
\sqrt{d_j}\right)=\bar C^{1/2}A\bar D^{1/2}\]
with a linear rate of convergence. Hence, it follows from \Cref{thm:twosidedinvert}
that limiting matrices $\bar C$ and $\bar D$ satisfy the 
necessary condition \Cref{eq:neccondtwosided} for a two-sided $\omega$-optimal diagonal preconditioner.
\end{theorem}

\begin{proof}
Let $\tilde P_k=\tilde A_k \bullet \tilde A_k$ and let $P_k =A_k \bullet A_k$.
It follows that $\tilde P_{k}=(A_{k-1}\bullet A_{k-1})\Diag(d_k)=P_{k-1}\Diag(d_k)$
where $(d_k)_{i}=1./\|(A_{k-1})_{:,i}\|^2$. Note $\|(A_{k-1})_{:,i}\|^2$
is exactly the $i$-th column sum of $P_{k-1}$
so $(d_k)_{i}=1./\sum_{l=1}^{n} (P_{k-1})_{l,i}$.
Moreover, also observe that 
$P_{k}=\Diag(c_k)(\tilde A_k \bullet \tilde A_k)=\Diag(c_k)\tilde P_k$.
Note $(c_k)_{i}=1./\sum_{l=1}^{n} (\tilde P_{k})_{i,l}$.
Hence, the sequences of iterates $\tilde P_k$ and $P_k$ can be viewed as
the iterates of Sinkhorn-Knopp Algorithm applied to $A\bullet A$. By
\cite{knight2008sinkhorn}, we know then that the sequence $P_k$ converges linearly to a matrix
$\bar P$ where $\bar P:=\bar C (A \bullet A) \bar D$ is 
a doubly stochastic matrix. 
The sequences of iterates $\tilde A_k$ and $A_k$  
are just the square roots of the iterates of the Sinkhorn-Knopp
algorithm applied to $A\bullet A$. Therefore, it follows from the above
argument that the sequence $A_k$ converges
linearly to a matrix $\bar A:=\bar C^{1/2}A\bar D^{1/2}$ that has column
and row norms all equal to $1$. The last conclusion of the theorem
follows from this observation and the 
necessary condition \Cref{eq:neccondtwosided} in \Cref{thm:twosidedinvert}.

\end{proof}

\subsection{Right-, Left-Sided Block Diagonal}
\label{sect:optblockdiag}

We now show that the above results extend directly to finding
\emph{block diagonal} preconditioners for least
squares solutions of full rank possibly overdetermined linear systems $Ax=b, A
\in \R^{m\times n}, m\geq n$. The preconditioner is extended from
diagonal to block diagonal (block upper-triangular). 
This relates to sparse QR preconditioning and
extends the results for $\omega$-optimal preconditioning with
partial Cholesky structure in~\cite[Theorem 2.7]{JungTorrWolkOmega2024}
and the block diagonal $\omega$-optimal preconditioner 
in~\cite[Prop. 3 part 3]{DoanW:11}.

\subsubsection{Right-Sided  Block Diagonal}
\label{sect:rightsidedblockdiag}
We let the linear transformation $B=$\textdef{$\Blkdiag(\cB)\in \Mn$}
denote the block diagonal matrix with blocks formed from the
set of square matrices $\cB=\{B_i\}_{i=1}^k$
of order $n_i, \sum_{i=1}^k n_i = n$. 
For our application we restrict the blocks to be of upper-triangular
structure denoted $\cR=\{R_i\}_{i=1}^k$, and see that the
best, with respect to the $\omega$ measure, comes
from Q-less QR decompositions of the corresponding blocks of $A$.

First we choose the number of blocks $k$ and the block
sizes $n_i, \sum_{i=1}^k n_i = n$. 
Thus we get the block structure
\[
A = \begin{bmatrix} A_1 & A_2 &\ldots & A_k \end{bmatrix},\quad
A / \Blkdiag(\cR) = 
\begin{bmatrix} A_1/ R_1 & A_2 / R_2&\ldots & A_k / R_k\end{bmatrix},
\]
where we use the \textsc{MATLAB} notation that $/$ denotes matrix 
division \textdef{$A/R=AR^{-1}$}.
Optimally, we want the sizes of the blocks chosen so that
the matrices $A_i^TA_i$ are chordal so that there is no loss of
sparsity. Moreover, we can allow a preliminary permutation of the
columns $A \leftarrow AP$ so that we can increase the 
effect of the preconditioner.

To extend the diagonal preconditioner results we solve:
\[
\min_\cB \{\omega ((AB)^T(AB)) : B = \Blkdiag(\cB) \}.
\]
The basic result we use follows.

\begin{prop}[{\cite[Prop. 3 part 3]{DoanW:11},\cite[Prop.
2.2]{KrukDoanW:10}}]
\label{prop:blckdiaf}
Let 
\[
A = \begin{bmatrix} A_1 & A_2 &\ldots & A_k \end{bmatrix},\quad
A_i \in \R^{m\times n_i}, \forall i,
\]
be a full rank $m\times n$ matrix, $m\geq n$.
Then an optimal block diagonal scaling 
\[
\cB:= \{B_1 , B_2 ,\ldots ,B_k\},\, B_i\in \R^{n_i\times n_i},\qquad
B:= \Blkdiag(\cB),
\]
that minimizes the measure $\omega$, i.e.,
\[
\min  \, \{\omega((A B)^T(A B)) : B = \Blkdiag(\cB) \},
\]
is found by satisfying the factorization
\[
B_iB_i^T = (A_i^TA_i)^{-1}, \quad i=1,\ldots,k.
\]
\end{prop}

\begin{corollary}
\phantomsection
\label{cor:partialqr}
Let $A,\cB,B$ be as in~\Cref{prop:blckdiaf}.
\\(i) If we restrict the diagonal blocks $B_i$ to be diagonal themselves, 
then we get the optimal diagonal preconditioner in~\Cref{prop:omegaoptDAD}.
\\(ii) If we restrict the diagonal blocks $B_i$ to be upper-triangular, 
then we get $B_i = R_i^{-1}, A_i = Q_iR_i$ from the QR decomposition of
$A_i$ (up to sign of rows of $R$).
\end{corollary}

\subsubsection{Left-Sided  Block Diagonal}
\label{sect:leftsidedblockdiag}

We continue with $A$ full column rank but with left-sided
preconditioning. We present a nonlinear equation that
characterizes optimality.
\begin{thm}\label{thm:leftblkgeneral}
Let $A\in\mathbb{R}^{m\times n}$ be full column rank with block
structure given by 
\[
A=\begin{bmatrix} \bar{A}_1 \cr \bar{A}_2 \cr \ldots \cr
\bar{A}_{\ell}\end{bmatrix},
  \bar{A}_j\in\mathbb{R}^{m_j \times n}, 
  \rank(\bar{A}_j) = m_j, \, \forall j.
\]
Then the  $\omega$-optimal left-sided block diagonal
preconditioner
\[
E := \Blkdiag(E_1,E_2,\ldots,E_{\ell}), \; E_j \in
\mathbb{R}^{m_j\times m_j},
\]
that minimizes the measure $\omega$, i.e., 
\[
\min \omega((EA)^TEA),
\]
is characterized by:
\[
\bar{A}_j\bar{A}_j^T =  \frac{1}{n} 
\left(\trace\sum_{t=1}^\ell 
\left(\bar{A}_t^T(E_t^TE_t)\bar{A}_t\right)
\right)
\, \bar{A}_j 
\left(A^T(E^TE)A\right)^{-1}
\bar{A}_j^T, \,\, \forall j=1,\ldots \ell.
\]
\end{thm}

\begin{proof}
To begin, we substitute $K:= E^TE$. That is, we solve 
\[
\min\{\omega(A^TKA): K = \Blkdiag(K_1,K_2,\ldots,K_{\ell}), \; K_j \in\mathbb{R}^{m_j\times m_j}\}.
\]
Notice that $A^TKA =  \sum_{j=1}^\ell \bar{A}_j^T K_j \bar{A}_j$. 
Define the linear transformation $\mathcal{G}(K) =  \sum_{j=1}^\ell \bar{A}_j^T K_j \bar{A}_j$. This yields
\[
\omega_{\mathcal{G}}(K) := \frac{ \trace(A^TKA)/n}{\det(A^TKA)^{1/n}} = \frac{\displaystyle\sum_{j=1}^{\ell} \trace(\bar{A}_j^T K_j \bar{A}_j)/n}{\det(A^TKA)^{1/n}}.
\]
Hence, for convenience and for the remainder of the proof define
\[
f(K):=\frac1n\sum_{j=1}^{\ell} \trace(\bar{A}_j^T K_j \bar{A}_j) \quad \text{and} \quad g(K):= \det(A^TKA)^{1/n}.
\]
The partial derivative of $g$ with respect to the block $K_j$ is given by
\begin{equation*}
\begin{aligned}
\langle \nabla_{K_j} g(K), \Delta K_j \rangle 
& = \langle \nabla \det(\mathcal{G}(K))^{1/n}, \Delta K_j \rangle  \\
&= \langle \frac{1}{n}\det (\mathcal{G}(K))^{\frac1n-1}\adj(\mathcal{G}(K)), \bar{A}_j^T \Delta K_j \bar{A}_j \rangle \\
&=  \frac{1}{n}  g(K) \langle  \mathcal{G}(K)^{-1}, \bar{A}_j^T  \Delta K_j \bar{A}_j \rangle  \\
&=  \frac{1}{n}  g(K) \langle  \bar{A}_j \mathcal{G}(K)^{-1} \bar{A}_j^T, \Delta K_j  \rangle .
\end{aligned}
\end{equation*}
Therefore, we have
\[
\nabla_{K_j} f(K) = \frac{1}{n} \bar{A}_j\bar{A}_j^T \quad \text{and} \quad \nabla_{K_j} g(K) = \frac{1}{n} g(K) \bar{A}_j (A^TKA)^{-1} \bar{A}_j.
\]
These in turn lead to the partial derivative of $\omega_{\mathcal{G}}$ with respect to $K_j$ given by
\[
\begin{aligned}
\nabla_{K_j} \omega_{\mathcal{G}}(K)  
& = \frac{1}{g(K)^2} (g(K) \nabla_{K_j} f(K) - f(K) \nabla_{K_j} g(K)) \\
& = \frac{1}{ng(K)} (\bar{A}_j\bar{A}_j^T  - f(K) \bar{A}_j (A^TKA)^{-1} \bar{A}_j^T) 
\end{aligned}
\]
Therefore 
\[
\nabla \omega_{\mathcal{G}} (K) = 0    \quad \Longleftrightarrow \quad \bar{A}_j\bar{A}_j^T =  \frac{1}{n} \sum_{t=1}^\ell \trace(\bar{A}_t\bar{A}_t^TK_t)\, \bar{A}_j (A^TKA)^{-1} \bar{A}_j^T, \quad \forall j.
\]
We can then recover the $E_j$ from $K$ using factorizations of the
blocks $K_j=E_j^TE_j$. 
\end{proof}

Finding an explicit solution for $E$ in~\Cref{thm:leftblkgeneral} is an
open question. However, if we
assume further that $A$ is a square matrix, the following
\Cref{cor:leftblksqure} connects the result from \Cref{cor:partialqr} to
this left sided preconditioning.

\begin{cor}\label{cor:leftblksqure}
	Let $A\in\Mn$ be nonsingular. Let the block of rows, $\bar{A}_j\in\mathbb{R}^{n_j \times n}$, be given by $A^T=[\bar{A}_1^T \, \bar{A}_2^T \, \ldots \, \bar{A}_{\ell}^T]$, $\sum_j n_j=n$. Define $\cE:=\{E_1,E_2,\dots,E_k\}$ by $E_i:= \bar R_i^{-T}$, where $\bar A_i^T:= \bar Q_i\bar R_i$ is the QR decomposition of $\bar A_i^T$. Then, $\cE$ is the optimal \underline{lower triangular} block diagonal scaling for
	\[
	\min  \, \{\omega((E A)^T(E A)) : E = \Blkdiag(E_1,E_2,\dots,E_k) \}.
	\]
\end{cor}
\begin{proof}
	Since $A$ and $E$ are square matrices, we observe 
	\[
	\omega\big((EA)^TEA\big) = \omega\big(A^TE^TEA\big) = \omega\big(EAA^{T}E^T\big),
	\]
	and thus \Cref{cor:partialqr} implies $E_i= \bar R_i^{-T}$, where $\bar A_i\bar A_i^T:= \bar R_i^T\bar R_i$ as defined above.
\end{proof}

\begin{remark}
	The result in \Cref{cor:leftblksqure} agrees with our
characterization  in \Cref{thm:leftblkgeneral}. Indeed, by defining $E_i:= \bar R_i^{-T}$ with $\bar A_i\bar A_i^T= \bar R_i^T\bar R_i$, we get that $\trace(\bar{A}_t\bar{A}_t^TE_t^TE_t) = \trace(\bar{R}_t^T\bar{R}_t\bar R_t^{-1}\bar R_t^{-T}) = \trace(I_{n_t}) = n_t$. Thus, 
	\[
	 \sum_{t=1}^\ell \trace(\bar{A}_t\bar{A}_t^TE_t^TE_t) =  \sum_{t=1}^\ell n_t = n.
	\]
	Now, define $\cI_j:=\begin{bmatrix}	0 & I_{n_j} & 0	\end{bmatrix}\in\R^{n_j\times n}$, with the identity $I_{n_j}\in\cM^{n_j}$ of order $n_j$, appropriately located such that $\cI_j A = \bar A_j$. Then,
	\[
	\begin{aligned}
		 \frac{1}{n} 
\left(\trace\sum_{t=1}^\ell 
\left(\bar{A}_t^TE_t^TE_t\bar{A}_t\right)
\right)
\bar{A}_j 
\left(A^TE^TEA\right)^{-1}
\bar{A}_j^T &= \bar{A}_j (A^TE^TEA)^{-1} \bar{A}_j^T\\
		&= \cI_j AA^{-1}E^{-1}E^{-T}A^{-T}A^T\cI_j^T\\
		&= \cI_j E^{-1}E^{-T}\cI_j^T\\
		&= E_j^{-1}E_j^{-T} = \bar R_j^T\bar R_j = \bar A_j\bar A_j^T.
	\end{aligned}
	\]
\end{remark}

\section{Computational Experiments}
\phantomsection
\label{sect:numerics}

Our computational experiments are organized into four parts, each
designed to evaluate different aspects of the proposed methods.

\noindent
In \Cref{Min Kappa Comparison}, we compare the computational efficiency of our
subgradient-based algorithms with existing SDP-based approaches for
computing approximate $\kappa$-optimal diagonal preconditioners.
\\In \Cref{PCG Linear Systems}, we assess the impact of these preconditioners on PCG performance for solving symmetric positive definite systems.
\\In \Cref{Two-Sided Empirics}, we study two-sided preconditioning for LSQR and
compare $\omega$‑optimal scaling with $\kappa$-optimal approaches.
\\Finally, in \Cref{sect:optkappimprovomega}, we investigate whether applying
$\omega$-optimal scaling to $\kappa$-optimally preconditioned systems can further improve PCG performance.

All experiments in \Cref{Min Kappa Comparison,PCG Linear Systems,Two-Sided Empirics,sect:optkappimprovomega}
are run on a 2023 Macbook Pro with an 8-core CPU and 128 GB of memory
using {\sc MATLAB} 2025b. 
The latest codes are available at this
\href{https://github.com/asujanani6/Optimal_Preconditioning}
{clickable-link} or with URL 
\url{https://github.com/asujanani6/Optimal_Preconditioning}.

\subsection{Minimizing $\kappa$ Efficiently}
\label{Min Kappa Comparison}

In this subsection, we study the practical performance
of our subgradient methods,
\Cref{Subgrad V,alg:GDeplustVv},
for computing approximate $\kappa$-optimal diagonal preconditioners.
We first compare the proposed methods with the SDP-based approaches of
\cite{gao2023scalable,doi:10.1287/opre.2022.0592}
on moderate-scale instances.
We then investigate the scalability of the proposed methods on substantially larger problems for which SDP-based approaches become impractical.

Throughout this subsection, condition numbers are evaluated using
MATLAB's {\sc eigs} routine with tolerance $10^{-10}$.
CPU times include all preprocessing and algorithmic overhead.
Unless otherwise stated, the default parameter settings from
\cite{gao2023scalable,doi:10.1287/opre.2022.0592}
are used.
\Cref{Subgrad V} uses tolerance $10^{-4}$,
maximum iteration limit $500$, sets stopcrit=$2|\kappa(d_{k+1})-\kappa(d_k)|/|\kappa(d_{k+1}))+\kappa(d_k)|$ at 
every iteration, and uses $\hat\delta=10^{-3}$.
\Cref{alg:GDeplustVv} sets a maximum of $80$ iterations and uses the same tolerance $10^{-4}$.

\subsubsection{Moderate Sizes}
\label{moderatesizekappamin}
In this subsection, we compare
our
\Cref{Subgrad V,alg:GDeplustVv},
with the SDP-based approaches of
\cite{gao2023scalable,doi:10.1287/opre.2022.0592}
on moderate-scale SuiteSparse and randomly generated instances.
The precise details of our line-search variant \Cref{alg:GDeplustVv}
are given in \Cref{Efficient_Subgradient_Method}.

Our methods yield an optimal diagonal matrix $D$ that minimizes
$\kappa(D^{1/2}MD^{1/2})$ for given $M\succ0$.
Convergence is proved for our \Cref{Subgrad V} that
often achieved stronger reductions in $\kappa$ than
\Cref{alg:GDeplustVv}; our line-search variant
\Cref{alg:GDeplustVv} was often slightly faster in practice 
than \Cref{Subgrad V}. In contrast,
the method in \cite{gao2023scalable}
solves an SDP over a prescribed subspace of diagonal preconditioners,
while \cite{doi:10.1287/opre.2022.0592} computes a right-scaling $E$ for $M$ 
by minimizing $\kappa((BE)^T(BE))$, with $M=B^TB$ for given nonsingular
$B$. Hence, when 
implementing this latter method, we take $B$ as a Cholesky factorization of $M$.

\Cref{table:OptKappaSuiteSparse}
reports results on eighteen matrices from the
SuiteSparse Matrix Collection \cite{davis2011university},
while \Cref{table:OptKappaRandom4}
considers randomly generated instances produced using
MATLAB's {\sc sprandsym} routine.

The proposed methods consistently achieved substantially larger
reductions in $\kappa$ than the SDP-based approaches.
On the SuiteSparse instances,
\Cref{Subgrad V} and \Cref{alg:GDeplustVv}
reduced $\kappa$ on average by $79.7\%$ and $68.5\%$,
respectively,
compared to $30.3\%$ and $21.1\%$
for \cite{doi:10.1287/opre.2022.0592}
and \cite{gao2023scalable}.
Moreover, both proposed methods were substantially faster,
with \Cref{alg:GDeplustVv}
often providing the best runtime performance.

For the random instances,
a time limit of 3600 seconds was imposed on
\cite{gao2023scalable}
due to its substantially higher computational cost.
Across all tested instances,
\Cref{Subgrad V} and \Cref{alg:GDeplustVv}
consistently reduced $\kappa$
more effectively and more efficiently than the competing methods.

\begin{table}[H]
{
\centering
\caption{Comparison of Algorithms for Min $\kappa$ on SuiteSparse}
\resizebox{\columnwidth}{!}
{
\input{table_1_Suite.tex}
}
\label{table:OptKappaSuiteSparse}
}
\end{table}

\begin{table}[H]
{
\centering
\caption{Comparison of Algorithms for Min $\kappa$ on Random}
\resizebox{\columnwidth}{!}
{
\input{table_2_Rand.tex}
}
\label{table:OptKappaRandom4}
}
\end{table}

\subsubsection{Large Sizes}
\label{Largesizeskappamin}
\Cref{table:LargeSuite,table:LargeRandom}
illustrate the scalability of
\Cref{Subgrad V,alg:GDeplustVv}
on substantially larger instances for which the SDP-based methods of
\cite{gao2023scalable,doi:10.1287/opre.2022.0592}
became computationally impractical.
\Cref{table:LargeSuite}
considers matrices from the SuiteSparse Matrix Collection,
while \Cref{table:LargeRandom}
considers randomly generated matrices with very large condition numbers.

On the SuiteSparse instances,
\Cref{Subgrad V}
and \Cref{alg:GDeplustVv}
reduced $\kappa$ by averages of $45.7\%$ and $41.9\%$,
respectively, while often requiring less than one minute.
On the random instances, which satisfy
$\kappa(M)>10^{11}$,
the corresponding average reductions were $11.6\%$ and $4.4\%$.

\begin{table}[H]
{
\tiny
\centering
\caption{Comparison of Algorithms for Min $\kappa$ on Large SuiteSparse}
\resizebox{\columnwidth}{!}
{
\input{table_3_Suite.tex}
}
\label{table:LargeSuite}
}
\end{table}

\begin{table}[H]
{
\tiny
\centering
\caption{Comparison of Algorithms for Min $\kappa$ on Large Random}
\label{table:LargeRandom}
\resizebox{\columnwidth}{!}
{
\input{table_4_Rand.tex}
}
}
\end{table}

In summary, the results show that the proposed subgradient-based methods
not only scale substantially better than existing SDP-based approaches,
but also achieve significantly larger reductions in $\kappa$
at a fraction of the computational cost.

\subsection{PCG Comparison for Solving Linear Systems}
\label{PCG Linear Systems}
In this subsection, 
we evaluate the effect of preconditioning on PCG 
performance for solving symmetric positive definite systems $Mx=b$.
We compare the diagonal preconditioners computed by 
\Cref{alg:GDeplustVv} with 
those obtained using the $\kappa$‑optimal scaling method from
\cite{doi:10.1287/opre.2022.0592}.
PCG is run with a relative residual tolerance of $10^{-4}$ and a maximum of $5\times 10^6$ iterations. 
The results are reported 
in \Cref{table:PCG1}. All test matrices are randomly 
generated using {\sc MATLAB}'s {\tt sprandsym} function with 
varying dimensions, densities, and $\kappa$ values.

\Cref{table:PCG1} shows that, over the 47 test instances, PCG 
requires on average $442{,}147$ iterations with the 
preconditioner from \Cref{alg:GDeplustVv}, compared to $549{,}002$ 
iterations with the preconditioner from \cite{doi:10.1287/opre.2022.0592}.
Averaged over all test instances, PCG with \Cref{alg:GDeplustVv} requires 
approximately $20\%$ fewer iterations while reducing total CPU time by nearly 
an order of magnitude compared with the method of \cite{doi:10.1287/opre.2022.0592}.

\begin{table}[H]
\tiny
{
\centering
\caption{PCG Comparison Using Preconditioners Found by Algorithm in \cite{doi:10.1287/opre.2022.0592} and \Cref{alg:GDeplustVv}}
\label{table:PCG1}
\resizebox{\columnwidth}{!}
{
\input{table_5_PCG.tex}
}
}
\end{table}

\subsection{LSQR Comparison: \Cref{alg:Two-Sided} vs. Two-Sided $\kappa$-optimal Scaling in \cite{doi:10.1287/opre.2022.0592}}\label{Two-Sided Empirics}
In \Cref{Two-Sided Empirics}, we compare our two-sided 
\Cref{alg:Two-Sided}, which decreases $\omega$ at every iteration, with the two-sided $\kappa$-optimal preconditioner 
of \cite{doi:10.1287/opre.2022.0592} for solving $\min_x\|b-Ax\|$ with LSQR.

The results are reported in \Cref{smallSuiteTwoSided,smallRandomTwoSided,TwoSidedLarge}. \Cref{smallSuiteTwoSided} and \Cref{smallRandomTwoSided} consider SuiteSparse and 
randomly generated matrices, respectively, with $n\leq 300$, since the method of \cite{doi:10.1287/opre.2022.0592} could not handle larger instances. In these experiments, LSQR is run with tolerance $10^{-8}$ and a maximum of $5000$ iterations. \Cref{TwoSidedLarge} considers SuiteSparse matrices with $n\geq 2000$ and compares LSQR with and without the preconditioner from \Cref{alg:Two-Sided}. Here, LSQR is run with tolerance $10^{-6}$ and a maximum of $100000$ iterations.

In all tables,
$
S:=(\hat C^{1/2}A\hat D^{1/2})^T(\hat C^{1/2}A\hat D^{1/2})$
and 
$T:=(D_1^{1/2}AD_2^{-1/2})^T(D_1^{1/2}AD_2^{-1/2})$
where $\hat C^{1/2},\hat D^{1/2}$ and $D_1^{1/2},D_2^{-1/2}$ are the preconditioners computed by \Cref{alg:Two-Sided} and \cite{doi:10.1287/opre.2022.0592}, respectively. Reported CPU times include both preconditioning and LSQR runtimes.

\begin{table}[H]
{
\centering
\caption{LSQR Comparison on Small Suitesparse Matrices: \Cref{alg:Two-Sided} vs \cite{doi:10.1287/opre.2022.0592}}
\resizebox{\columnwidth}{!}
{
\input{table_6_Suite.tex}
}
\label{smallSuiteTwoSided}
}
\end{table}

\begin{table}[H]
{
\centering
\caption{LSQR Comparison on Small Random Matrices: \Cref{alg:Two-Sided} vs \cite{doi:10.1287/opre.2022.0592}}
\resizebox{\columnwidth}{!}
{
\input{table_7_Rand.tex}
}
\label{smallRandomTwoSided}
}
\end{table}

\begin{table}[H]
{
\centering
\caption{LSQR Comparison on Large Suitesparse Matrices: \Cref{alg:Two-Sided} vs No Preconditioning}
\resizebox{\columnwidth}{!}
{
\input{table_8_Suite.tex}
}
\label{TwoSidedLarge}
}
\end{table}
The results in \Cref{smallSuiteTwoSided,smallRandomTwoSided} show that 
although the two-sided 
$\kappa$‑optimal preconditioner in 
\cite{doi:10.1287/opre.2022.0592} often achieves a
stronger reduction in $\kappa$, \Cref{alg:Two-Sided} consistently 
yields lower 
$\omega$ values
at a significantly lower computational cost. Importantly, reductions in
$\omega$
correlate more strongly with LSQR iteration counts, resulting in faster
convergence on the majority of tested instances.

In summary, 
these experiments reinforce that $\omega$‑optimal preconditioning offers a more
practical and computationally efficient alternative to $\kappa$‑optimal scaling
for LSQR, particularly in large-scale and nonsymmetric settings.


\subsection{From $\kappa$-optimal $M$ to ``Improve PCG" using
$\omega$-Optimal Scaling}
\label{sect:optkappimprovomega}
We study the effect of applying $\omega$-optimal (Jacobi) diagonal scaling to a $\kappa$-optimally scaled matrix $M\succ0$.\footnote{For the large problems in \Cref{table:optkappaLarge} we used the fastlinux server, cpu155.math.private, a Dell PowerEdge R650 with two Intel Xeon Gold 6334 8-core 3.6 GHz (Ice Lake) processors and 256 GB RAM.}
To this end, we construct the matrix $M$ using \Cref{thm:optcondkappa}: the extreme eigenpairs 
satisfy~\cref{eq:optkappascaled,eq:optkappascaledsigns}, while the 
remaining eigenvalues are spread uniformly in $(\lambda_n,\lambda_1)$ with orthonormal eigenvectors 
chosen in the orthogonal complement of span$\{x_1,x_n\}$. The remaining eigenvectors are generated via a 
sparse QR factorization, so the density of $M$ cannot be prescribed exactly in advance. A more detailed description of 
the construction is provided in our code.

After constructing $M$, we apply the $\omega$-optimal (Jacobi) scaling
$
J=D^{1/2}MD^{1/2}$ 
where $D=\Diag(1./\diag(M))
$,
and compare PCG applied to $Mx=b$ and to the scaled system $J(D^{-1/2}x)=D^{1/2}b$. Iteration counts and runtimes are averaged over five initial points. The results are reported in \Cref{table:optkappaMedium,table:optkappaLarge}.

\begin{table}[H]
{
\tiny
\centering
\caption{PCG tol. $1e\!-\!7$; Medium;
$\kappa$-opt \underline{\bf VS} $J, \omega$-opt of $M$}
\resizebox{\columnwidth}{!}
{
\input{table_9_PCG.tex}
}
\label{table:optkappaMedium}
}
\end{table}

\begin{table}[H]
{
\tiny
\centering
\caption{PCG tol. $1e\!-\!7$; Large;
$\kappa$-opt \underline{\bf VS} $J, \omega$-opt of $M$}
\resizebox{\columnwidth}{!}
{
\input{table_10_PCG_Linux.tex}
}
\label{table:optkappaLarge}
}
\end{table}
The experiments
show
that
surprisingly, despite increasing the $\kappa$‑condition number, the $\omega$‑optimal 
(Jacobi) scaling consistently yields dramatic improvements in PCG 
convergence. 
Across all medium and large instances 
presented
in \Cref{table:optkappaMedium,table:optkappaLarge},
PCG requires between 96–98.6\% fewer iterations and correspondingly less CPU 
time after $\omega$‑optimal scaling.
This improvement is especially pronounced
on the large-scale instances presented in \Cref{table:optkappaLarge}.

In summary, our computational experiments in \Cref{sect:numerics} provide compelling evidence that the 
$\omega$‑condition number is a stronger predictor of iterative solver performance 
than $\kappa$, and that $\omega$‑optimal preconditioners strike a superior balance 
between computational cost and practical effectiveness.

\section{Conclusion}
\label{sect:concl}
In this paper, we studied optimal diagonal preconditioning 
through the lens of two condition numbers: the classical $\kappa$ and the 
averaging-based $\omega$-condition number. On the theoretical side, we introduced an 
affine-based pseudoconvex reformulation of the $\kappa$-optimal 
preconditioning problem, yielding simple optimality conditions and 
enabling efficient optimization over an $n$-dimensional vector. 
Moreover, since the $\kappa$-condition number is positively homogeneous 
of degree zero, the associated minimization problem is Hadamard ill-posed. 
To address this, we introduce a reformulation that removes this homogeneity, 
leading to improved computational stability in practice. 
Building on this formulation, we develop a highly efficient subgradient 
method with convergence guarantees that significantly outperforms existing SDP-based approaches in both scalability and accuracy.
For example, \Cref{table:OptKappaSuiteSparse} shows 
that \Cref{Subgrad V} (resp. \Cref{alg:GDeplustVv}) is, on 
average, 77.48 (resp. 356.22) times faster 
than \cite{doi:10.1287/opre.2022.0592} when 
minimizing $\kappa$ on relatively small SuiteSparse matrices. 
Our methods also substantially outperform \cite{gao2023scalable} 
in terms of runtime.
In addition to these speedups, our subgradient methods 
consistently achieve significantly greater reductions 
in $\kappa$ than both \cite{doi:10.1287/opre.2022.0592} 
and \cite{gao2023scalable}, often reducing it by more than half. 
Finally, our algorithms scale to large problem instances with
 hundreds of thousands of variables, solving them within minutes,
whereas existing 
methods \cite{doi:10.1287/opre.2022.0592,gao2023scalable} 
struggle to handle such instances within a reasonable time.


In parallel, we provided explicit and 
unified characterizations of $\omega$-optimal diagonal and block-diagonal 
preconditioners, showing that many classical schemes such as Jacobi scaling, 
row/column normalization, and matrix balancing arise naturally as $\omega$-optimal 
solutions or stationary points of the $\omega$-optimal
preconditioning problem. These results offer a new perspective on widely 
used preconditioning techniques and, to the best of our knowledge, 
constitute the first comprehensive comparison of $\kappa$- and $\omega$-based 
optimality conditions.

Our numerical results further reveal a clear and practically important 
message: while $\kappa$-optimal preconditioners reduce the worst-case 
condition number more aggressively, $\omega$-optimal preconditioners are 
substantially cheaper to compute and more strongly correlated with the 
performance of iterative methods such as PCG and LSQR. Moreover, 
applying the $\omega$-optimal diagonal scaling to linear systems that are already $\kappa$-optimally 
preconditioned leads to further significant improvements in PCG's performance.

Overall, our findings suggest that $\omega$-based preconditioning provides a 
computationally efficient and practically superior alternative to $\kappa$-based 
approaches for large-scale problems, and highlight the importance of moving 
beyond worst-case conditioning when designing preconditioners.

\vspace{3em}
\noindent{\Large\bfseries Statements and Declarations}

\vspace{1em}

\noindent\textbf{Funding.}
The first author acknowledges the support of the Natural Sciences and Engineering
Research Council of Canada (NSERC) [Application ID: RGPIN-2021-02644].
The second and fifth authors acknowledge support from Natural Sciences and Engineering
Research Council of Canada (NSERC)  [Work Order Number: 50503-11728 2925 105].
The third author was partially supported by NSERC [Application ID: RGPIN-2021-02644],
NSERC [Work Order Number: 50503-11728 2925 105], the Department of Combinatorics
at the University of Waterloo, and Professor Walaa Moursi.
The fourth author was partially supported by grant PID2022-136399NB-C21
funded by ERDF/EU and by MICIU/AEI/10.13039/501100011033; by Centro de
Modelamiento Matemático (CMM) BASAL fund FB210005 for center of
excellence; and Fondecyt Postdoctorado 3250039 from ANID Chile.

\vspace{0.8em}

\noindent\textbf{Competing Interests.}
The authors have no relevant financial or non-financial interests to disclose.

\vspace{0.8em}

\noindent\textbf{Data Availability.}
The datasets and matrices generated and analyzed during the current study
are publicly available in the GitHub repository:
\url{https://github.com/asujanani6/Optimal_Preconditioning}.

\vspace{0.8em}

\noindent\textbf{Code Availability.}
The full code used to generate the results and experiments in this paper
is publicly available in the GitHub repository:
\url{https://github.com/asujanani6/Optimal_Preconditioning}.

\newpage
\appendix

\section{Assumptions and Technical Results from \cite{kiwiel2001convergence}}\label{Kiweil assumptions}
Let $f:\mathbb R^{n}\rightarrow \bar {\mathbb R}:=\mathbb R \cup \left\{-\infty,\infty\right\}$ be an extended real-valued 
function with \textdef{domain $\bar D$}. 
The following minimization problem is considered in \cite{kiwiel2001convergence}:
\begin{equation}\label{Kiwiel optimization problem}
f_{*}:=\inf\left\{f(x):x\in X\right\},
\end{equation}
with the following assumptions on $f$ and $X$:
\index{$\bar D$, domain}
\begin{itemize}
\item[\textbf{A1}]: $\Int \bar D$ is nonempty and convex.
\item[\textbf{A2}]: ${\overline{\lim}}_{t\downarrow 0} f(x+t(y-x))\leq f(x)$, $\forall x \in \bar D, y \in \Int \bar D$. 
\item[\textbf{A3}]: $f$ is upper semicontinuous (usc) on $\Int \bar D$, i.e., $f(x)=\lim_{\epsilon \downarrow 0}\sup_{B(x,\epsilon)}f, \quad \forall x \in \Int\bar D$
where $B(x,\epsilon):=\left\{y:\|y-x\|\leq \epsilon \right\}$.
\item[\textbf{A4}]: $f$ is quasiconvex  on $\Int\bar D$, i.e., the set $\{x \in \Int\bar D: f(x)\leq \alpha\}$ is convex $\forall \alpha \in \mathbb R$.
\item[\textbf{A5}]: The constraint set $X \subset \mathbb R^{n}$ is closed convex, and $X\cap \Int\bar D\neq \emptyset$.
\end{itemize}

The following lemma from \cite{kiwiel2001convergence}
discusses that fractional programs with certain structures are quasiconvex and also
provides one characterization of their quasisubgradients.

\begin{lemma}[{\cite[Lemma 4]{kiwiel2001convergence}}]\label{Lemma Subgrad Sublevel}
Suppose $f(x)=a(x)/b(x)$ for all $x \in \Int\bar D$, where $a$ is a convex function, $b$ is finite
and positive on $\Int\bar D$, $\Int\bar D$ is convex, and one of
the following conditions holds:
\begin{itemize}
	\item[(a)] $b$ is affine;
	\item[(b)] $a$ is nonnegative on $\Int\bar D$ and $b$ is concave;
	\item[(c)] $a$ is nonpositive on $\Int\bar D$ and $b$ is convex.
\end{itemize}
Then $f$ is quasiconvex on $\Int\bar D$ and for each $x \in \Int\bar D$, if $\alpha:=f(x)$ is finite then $a-\alpha b$ is convex and $\partial[a-\alpha b](x)\subset {\bar{\partial}}^{\circ} f(x)$.
\end{lemma}

The following lemma can be found in \cite{kiwiel2001convergence} and is useful for showing that assumptions (A1)-(A4) hold for our set-up in 
\Cref{eq:DClosedSet} since $\kappa(\cDo(d))$ is the ratio of a convex and concave function that is positive on $\Omega$.

\begin{lemma}[{\cite[Remark 2]{kiwiel2001convergence}}]\label{Remark Kiwiel Fractional}
Suppose $a$ and $\tilde b:=-b$ are proper convex functions on $\Int \bar D^{a} \cap \Int \bar D^{b}$, $a$
is nonnegative on the (open and convex set) $C:=\left\{y \in \Int \bar D^{a} \cap \Int \bar D^{b}: b(y)>0 \right\} \neq \emptyset$,
and
\[
f(x):=
\begin{cases}
a(x)/b(x), \quad \text{ if } x\in C,\\
\sup_{y\in C} \overline{\lim}_{t\downarrow 0} f(x+t(y-x)),  \quad \text{ if } x\in \textrm{bd } C,\\
\infty,  \quad \text{ if } x\notin \clo C
\end{cases}
\]
where \textdef{$\clo C$} denotes the closure,
\textrm{bd} $C$ denotes the boundary of $C$, and $\bar D^{a}$ and $\bar D^{b}$ 
denote the domains of $a$ and $b$, respectively. Then assumptions \textbf{A1}-\textbf{A4} hold with $\Int \bar D=C$.
\end{lemma}

\section{An Efficient Line-Search Subgradient Method}
\label{Efficient_Subgradient_Method}
This section presents an efficient line-search subgradient method
for minimizing $\kappa(v):=\kappa(\cVo(v))$ where $\cVo(v)$ is as in \Cref{non homogoenous function}.

Let $w = e+Vv \in \Rnpp$ correspond to the current iterate $v$. Also, let $\sigma$ be
a small scalar between $0$ and $1$ and let $\Delta w =
V\Delta v$ where $\Delta v=-\nabla\kappa(v)$ where $\kappa(v)$ is as in \Cref{Not Homogoenous Formulation}.
From~\Cref{lem:KposdefInertia}, we can use a ratio test and guarantee
positive definiteness from
$w + tV\Delta v  = w +t\Delta w> \sigma w$, or equivalently
$t\Delta w> (\sigma-1) w$. Therefore, we conclude that the maximum step
with safeguarding, so as not to get too close to the boundary, is
\begin{equation}
\label{eq:tmax}
t_{\max}(\sigma) \cong t_{\max} :=  \min\left\{
\frac {(\sigma-1)w_i}{\Delta w_i} : \Delta w_i < 0 \right\}  > 0.
\end{equation}

Recall the gradient $\nabla \kappa(v) =
\nabla(\kappa \circ \cVo(v))$ in \Cref{cor:optandgradcVo}.
Our
line search in \Cref{alg:GDeplustVv} maintains: 
\begin{linesrch}[for \Cref{alg:GDeplustVv}]
\label{ln:linesalgGDeff}
Backtrack and maintain:
(i) $t \leq t_{\max}$
from \cref{eq:tmax}; (ii) monotonic nonincrease of
the objective $\kappa(v+t\Delta v)$; and nonpositivity of the directional
derivative $\nabla \kappa(v+t\Delta v)^T\Delta v\leq 0$. 
\end{linesrch}
\Cref{alg:GDeplustVv} is presented below. 
 \begin{algorithm}[H]
\caption{An Efficient Line-Search Subgradient Method for Minimizing $\kappa(v)$.}
\label{alg:GDeplustVv}
\begin{algorithmic}[1]
\Require A symmetric positive definite matrix $A \succ 0$, a max iteration count
mainmaxiter,
stopping tolerances \textrm{maintoler}$>0$ and \textrm{linesrchtoler}$>0$, and a small
scalar $0<\sigma\ll 1$.
\Statex set $k \gets 0$, $v_1=0 \in \mathbb R^{n-1}$, $w_1=e_n \in \mathbb R^{n}$, and $t_k \gets \infty$;
\Statex compute a min eigenpair $(\lambda_n^{1},x_n^{1})$ and max eigenpair $(\lambda_1^{1},x_1^{1})$ of $A$;
\Statex compute $\kappa(v_1)$ and $\nabla \kappa(v_1)$ according to \Cref{Not Homogoenous Formulation} and \Cref{gradient of V}, respectively;
\Statex set \textrm{relnormg}$=\|\nabla \kappa(v_1)\|^2/(1+\kappa(v_1 )^2)$;
\While{\textrm{relnormg} $>$ \textrm{maintoler} \,\&\,  $t_k>$ \textrm{linesrchtoler} \,\&\,
$k\leq$ \textrm{mainmaxiter}} (main outer loop)
\State  set $k\leftarrow k+1$;
\State set $\Delta v_{k}\leftarrow -\nabla \kappa (v_k)$;
\State 
use \cref{eq:tmax} with $\Delta w:=V\Delta v_{k}$ and $w:=w_k$ to evaluate
$t_{\max}(\sigma)$;
\State
\label{LineSearchLabe}
perform~\Cref{ln:linesalgGDeff} to find $t_k$;
\State set
$v_{k+1}=v_k+t_k \Delta v_k$ and $w_{k+1}=e_n+Vv_{k+1}$;
\State compute min eigenpair $(\lambda_n^{k+1},x_n^{k+1})$ and max eigenpair $(\lambda_1^{k+1},x_1^{k+1})$ of $A\Diag(w_{k+1})$;
\State compute $\kappa(v_{k+1})$ and $\nabla \kappa(v_{k+1})$ according
to  \Cref{Not Homogoenous Formulation} and \Cref{gradient of V}, respectively;
\State update \textrm{relnormg}$=\|\nabla \kappa(v_{k+1})\|^2/(1+\kappa(v_{k+1} )^2)$;

\EndWhile  (main outer loop)

\Ensure $\hat D:=\Diag(w_{k+1})$.

\end{algorithmic}
\end{algorithm}
In practice, we take \textrm{linesearchtol}
to be $10^{-7}$ and \textrm{maintoler} to be $10^{-4}$.

\section{List of SuiteSparse Matrices Used in Experiments}\label{Appendix SuiteSparse}
\Cref{table:OptKappaSuiteSparse} considers the following matrices:
\begin{itemize}[label={}]
	\item ``bcsstk08.mat'', ``bcsstk13.mat'', ``bcsstk21.mat'', ``bcsstk23.mat'', ``bcsstk24.mat''
	\item ``bcsstk26.mat'', ``bcsstk28.mat'', ``bcsstk34.mat'', ``494\_bus.mat'', ``662\_bus.mat''
	\item  ``nasa1824.mat'', ``nasa2146.mat'', ``nasa2910.mat'', ``nos1.mat'', ``nos2.mat'', ``nos4.mat''
	\item  ``nos5.mat'', ``nos7.mat''
\end{itemize}

\Cref{table:LargeSuite} considers the following matrices:
\begin{itemize}[label={}]
	\item ``Pres\_Poisson.mat'', ``bcsstk25.mat'', ``bcsstm25.mat'', ``gyro\_m.mat''
	\item ``gyro.mat'', ``bcsstk36.mat'', ``wathen100.mat'', ``wathen120.mat'', ``minsurfo.mat''
	\item  ``gridgena.mat'', ``G2\_circuit.mat''
\end{itemize}

\Cref{smallSuiteTwoSided} considers the following matrices:
\begin{itemize}[label={}]
	\item ``08blocks.mat'', ``arc130.mat'', ``bwm200.mat'', ``ck104.mat'', ``ex1.mat''
	\item ``football.mat'', ``gre\_185.mat'', ``gre\_216a.mat'', ``impcol\_a.mat'', ``impcol\_c.mat''
	\item  ``impcol\_e.mat'', ``jazz.mat'', ``lshp\_265.mat'', ``olm100.mat'', ``polbooks.mat''
	\item   ``rajat11.mat'', ``robot.mat'', ``rotor1.mat'', ``rw136.mat'', ``tub100.mat''
	\item   ``utm300.mat'', ``west013.mat''
\end{itemize}

\Cref{TwoSidedLarge} considers the following matrices:
\begin{itemize}[label={}]
	\item ``airfold\_2d.mat'', ``c-24.mat'', ``c-36.mat'', ``c-39.mat'', ``cavity21.mat''
	\item ``conf5\_-4x4-18.mat'', ``coupled.mat'', ``delaunay\_n11.mat'', ``delaunay\_n13.mat'', ``epb1.mat''
	\item  ``ex24.mat'', ``G34.mat'', ``G36.mat'', ``G59.mat'', ``G63.mat''
	\item   ``G64.mat'', ``garon1.mat'', ``Hamrle2.mat'', ``jan99jac040sc.mat'', ``Na5.mat''
	\item   ``t2d\_q9.mat'', ``tols4000.mat'', ``utm5940.mat'', ``viscoplastic1.mat''
	\item   ``whitaker3.mat''
\end{itemize}

\cleardoublepage
\phantomsection
\addcontentsline{toc}{section}{Index}
\printindex
\label{ind:index}

\cleardoublepage
\phantomsection
\addcontentsline{toc}{section}{Bibliography}
\bibliographystyle{siam}
\bibliography{master,psd,edm,bjorBOOK,davidt,leo}
\label{bib:bibl}

\end{document}

%% file: table_1_Suite.tex
\begin{tabular}{|ccc|cccc|cccc|cc|} \hline
\multicolumn{3}{|c||}{Dim \& Density, \& $\kappa(M)$} & \multicolumn{4}{|c|}{\% Reduction in $\kappa$} & \multicolumn{4}{|c|}{CPU Time (seconds)} &\multicolumn{2}{|c|}{CPU Ratios}\cr\cline{1-3}\cline{4-13}
  $n$ & density & $\kappa(M)$ & \cite{doi:10.1287/opre.2022.0592} & \cite{gao2023scalable} & \Cref{Subgrad V} & \Cref{alg:GDeplustVv} & \cite{doi:10.1287/opre.2022.0592} & \cite{gao2023scalable} & \Cref{Subgrad V} & \Cref{alg:GDeplustVv} & \cite{doi:10.1287/opre.2022.0592}/\Cref{Subgrad V} & \cite{doi:10.1287/opre.2022.0592}/\Cref{alg:GDeplustVv} \cr\hline
 1074 & 1.1e-02 & 2.6e+07 & 0.0e+00 & 0.0e+00 & 9.9e+01 & 9.9e+01 & 48.203 & 51.700 & 8.469 & 5.644 &  5.7 &  8.5 \cr\hline
 2003 & 2.1e-02 & 1.1e+10 & 0.0e+00 & 0.0e+00 & 9.8e+01 & 6.8e+01 & 330.887 & 324.334 & 12.968 & 0.120 & 25.5 & 2752.3 \cr\hline
 3600 & 2.1e-03 & 1.8e+07 & 0.0e+00 & 0.0e+00 & 6.0e+01 & 6.0e+01 & 510.244 & 1102.270 & 24.644 & 9.538 & 20.7 & 53.5 \cr\hline
 3134 & 4.6e-03 & 2.6e+12 & -9.7e-04 & 0.0e+00 & 7.5e+01 & 7.0e+01 & 178.360 & 1347.551 & 2.542 & 0.410 & 70.2 & 435.3 \cr\hline
 3562 & 1.3e-02 & 1.9e+11 & 0.0e+00 & 0.0e+00 & 8.6e+01 & 7.8e+01 & 1621.648 & 948.059 & 36.940 & 1.694 & 43.9 & 957.4 \cr\hline
 1922 & 8.2e-03 & 1.7e+08 & 0.0e+00 & 0.0e+00 & 9.6e+01 & 8.8e+01 & 182.146 & 448.809 & 4.524 & 0.405 & 40.3 & 450.0 \cr\hline
 4410 & 1.1e-02 & 9.5e+08 & 0.0e+00 & 8.7e+01 & 9.6e+01 & 8.3e+01 & 3170.182 & 2602.683 & 4.481 & 2.784 & 707.5 & 1138.7 \cr\hline
  588 & 6.2e-02 & 2.8e+04 & 9.9e+01 & 0.0e+00 & 9.0e+01 & 9.4e+01 & 19.195 & 7.152 & 0.463 & 1.379 & 41.5 & 13.9 \cr\hline
  494 & 6.8e-03 & 2.4e+06 & 9.0e+01 & 0.0e+00 & 9.1e+01 & 3.3e+01 & 9.222 & 2.959 & 0.832 & 0.116 & 11.1 & 79.8 \cr\hline
  662 & 5.6e-03 & 7.9e+05 & 9.5e+01 & 0.0e+00 & 7.4e+01 & 4.5e+01 & 16.293 & 5.947 & 0.324 & 0.203 & 50.2 & 80.3 \cr\hline
 1824 & 1.2e-02 & 1.9e+06 & 0.0e+00 & 0.0e+00 & 8.7e+01 & 7.1e+01 & 182.642 & 165.698 & 3.468 & 1.958 & 52.7 & 93.3 \cr\hline
 2146 & 1.6e-02 & 1.7e+03 & 0.0e+00 & 5.3e+01 & 5.5e+01 & 4.9e+01 & 428.711 & 137.240 & 2.604 & 16.508 & 164.6 & 26.0 \cr\hline
 2910 & 2.1e-02 & 6.0e+06 & 0.0e+00 & 0.0e+00 & 9.1e+01 & 7.5e+01 & 619.758 & 613.440 & 8.306 & 3.731 & 74.6 & 166.1 \cr\hline
  237 & 1.8e-02 & 2.0e+07 & 2.1e+01 & 5.5e+00 & 7.9e+01 & 6.5e+01 & 2.178 & 0.620 & 0.276 & 0.056 &  7.9 & 38.8 \cr\hline
  957 & 4.5e-03 & 5.1e+09 & 2.8e+01 & 3.5e+01 & 6.0e+01 & 6.5e+01 & 10.107 & 8.676 & 6.419 & 0.145 &  1.6 & 69.7 \cr\hline
  100 & 5.9e-02 & 1.6e+03 & 3.8e+01 & 3.1e+01 & 3.5e+01 & 3.5e+01 & 0.800 & 0.297 & 0.036 & 0.970 & 22.2 &  0.8 \cr\hline
  468 & 2.4e-02 & 1.1e+04 & 8.3e+01 & 7.9e+01 & 7.7e+01 & 7.1e+01 & 15.428 & 2.957 & 1.081 & 1.298 & 14.3 & 11.9 \cr\hline
  729 & 8.7e-03 & 2.4e+09 & 9.2e+01 & 8.9e+01 & 8.6e+01 & 8.4e+01 & 39.553 & 6.439 & 0.987 & 1.108 & 40.1 & 35.7 \cr\hline
\end{tabular}

%% file: table_2_Rand.tex
\begin{tabular}{|ccc|cccc|cccc|cc|} \hline
\multicolumn{3}{|c||}{Dim \& Density, \& $\kappa(A)$} & \multicolumn{4}{|c|}{\% Reduction in $\kappa$} & \multicolumn{4}{|c|}{CPU Time (seconds)} &\multicolumn{2}{|c|}{CPU Ratios}\cr\cline{1-3}\cline{4-13}
  $n$ & density & $\kappa(M)$ & \cite{doi:10.1287/opre.2022.0592} & \cite{gao2023scalable} & \Cref{Subgrad V} & \Cref{alg:GDeplustVv} & \cite{doi:10.1287/opre.2022.0592} & \cite{gao2023scalable} & \Cref{Subgrad V} & \Cref{alg:GDeplustVv} & \cite{doi:10.1287/opre.2022.0592}/\Cref{Subgrad V} & \cite{doi:10.1287/opre.2022.0592}/\Cref{alg:GDeplustVv} \cr\hline
 2000 & 1.0e-03 & 2.0e+06 & -4.8e+01 & 4.8e+01 & 5.3e+01 & 6.6e+01 & 14.498 & 76.446 & 4.470 & 1.154 &  3.2 & 12.6 \cr\hline
 2500 & 8.1e-04 & 2.5e+06 & -5.0e+01 & 2.9e+01 & 5.2e+01 & 5.2e+01 & 21.500 & 229.744 & 9.068 & 1.233 &  2.4 & 17.4 \cr\hline
 3000 & 6.7e-04 & 3.0e+06 & -1.8e+01 & 0.0e+00 & 1.9e+01 & 1.3e+01 & 32.757 & 358.866 & 5.011 & 0.476 &  6.5 & 68.8 \cr\hline
 3500 & 5.8e-04 & 3.5e+06 & -1.9e+01 & 2.5e+01 & 4.3e+01 & 4.9e+01 & 43.032 & 565.527 & 8.757 & 1.961 &  4.9 & 21.9 \cr\hline
 4000 & 5.0e-04 & 4.0e+06 & -4.4e+01 & 0.0e+00 & 2.8e+01 & 6.3e+00 & 56.246 & 631.303 & 11.281 & 0.519 &  5.0 & 108.4 \cr\hline
 4500 & 4.5e-04 & 4.5e+06 & -3.0e+01 & 0.0e+00 & 3.3e+01 & 4.2e+01 & 69.000 & 888.172 & 3.774 & 2.202 & 18.3 & 31.3 \cr\hline
 5000 & 4.0e-04 & 5.0e+06 & -4.4e+01 & 0.0e+00 & 3.9e+01 & 3.9e+01 & 85.069 & 2009.496 & 13.994 & 2.341 &  6.1 & 36.3 \cr\hline
 5500 & 3.7e-04 & 5.5e+06 & -3.5e+01 & 1.4e+01 & 2.8e+01 & 2.8e+01 & 104.179 & 2512.849 & 2.763 & 2.614 & 37.7 & 39.8 \cr\hline
 6000 & 3.3e-04 & 6.0e+06 & -4.3e+01 & 0.0e+00 & 3.0e+01 & 3.4e+01 & 127.060 & 3680.562 & 5.383 & 2.748 & 23.6 & 46.2 \cr\hline
 6500 & 3.1e-04 & 6.5e+06 & -2.9e+01 & 1.9e+01 & 2.7e+01 & 3.2e+01 & 146.476 & 3755.564 & 4.694 & 3.124 & 31.2 & 46.9 \cr\hline
 7000 & 2.9e-04 & 7.0e+06 & -2.9e+01 & 2.4e+01 & 3.3e+01 & 3.0e+01 & 168.726 & 3838.763 & 11.441 & 3.243 & 14.7 & 52.0 \cr\hline
 7500 & 2.7e-04 & 7.5e+06 & -4.1e+01 & 2.5e+01 & 3.2e+01 & 2.8e+01 & 193.993 & 3879.417 & 12.929 & 3.649 & 15.0 & 53.2 \cr\hline
 8000 & 2.5e-04 & 8.0e+06 & -1.9e+01 & 3.0e+01 & 3.2e+01 & 2.7e+01 & 215.347 & 3790.944 & 16.234 & 3.645 & 13.3 & 59.1 \cr\hline
 8500 & 2.4e-04 & 8.5e+06 & -2.8e+01 & 3.6e+01 & 2.5e+01 & 2.6e+01 & 246.213 & 3929.044 & 8.086 & 4.716 & 30.5 & 52.2 \cr\hline
 9000 & 2.2e-04 & 9.0e+06 & -1.5e+01 & -9.9e+00 & 9.1e+00 & 3.2e+00 & 271.863 & 3903.914 & 7.354 & 0.931 & 37.0 & 292.1 \cr\hline
 9500 & 2.1e-04 & 9.5e+06 & -4.2e+01 & 3.6e+01 & 2.6e+01 & 2.3e+01 & 323.748 & 4147.347 & 13.329 & 5.681 & 24.3 & 57.0 \cr\hline
10000 & 2.0e-04 & 1.0e+07 & -3.2e+01 & 1.2e+01 & 2.2e+01 & 1.8e+01 & 341.504 & 4304.042 & 8.240 & 5.232 & 41.4 & 65.3 \cr\hline
10500 & 1.9e-04 & 1.0e+07 & -1.2e+01 & 0.0e+00 & 2.7e+01 & 2.1e+01 & 384.690 & 4302.243 & 20.255 & 6.627 & 19.0 & 58.0 \cr\hline
11000 & 1.8e-04 & 1.1e+07 & -4.1e+01 & 3.0e+00 & 2.2e+01 & 1.9e+01 & 463.158 & 4273.218 & 11.945 & 6.085 & 38.8 & 76.1 \cr\hline
11500 & 1.7e-04 & 1.1e+07 & -3.4e+01 & 0.0e+00 & 2.3e+01 & 2.0e+01 & 487.216 & 4000.550 & 9.915 & 6.705 & 49.1 & 72.7 \cr\hline
12000 & 1.7e-04 & 1.2e+07 & -3.0e+01 & 0.0e+00 & 2.2e+01 & 1.9e+01 & 511.686 & 4581.828 & 10.963 & 6.709 & 46.7 & 76.3 \cr\hline
12500 & 1.6e-04 & 1.3e+07 & -3.1e+01 & 0.0e+00 & 4.0e+00 & 3.2e+00 & 546.856 & 4059.532 & 3.496 & 2.080 & 156.4 & 263.0 \cr\hline
13000 & 1.5e-04 & 1.3e+07 & -2.9e+01 & 0.0e+00 & 2.4e+01 & 1.8e+01 & 593.252 & 4586.045 & 40.099 & 7.383 & 14.8 & 80.4 \cr\hline
13500 & 1.5e-04 & 1.4e+07 & -3.0e+01 & 0.0e+00 & 2.2e+01 & 1.8e+01 & 665.554 & 5145.811 & 15.941 & 8.300 & 41.8 & 80.2 \cr\hline
14000 & 1.4e-04 & 1.4e+07 & -3.0e+01 & 0.0e+00 & 2.2e+01 & 1.7e+01 & 719.336 & 4122.784 & 14.430 & 7.864 & 49.9 & 91.5 \cr\hline
14500 & 1.4e-04 & 1.4e+07 & -3.4e+01 & 0.0e+00 & 2.1e+01 & 1.2e+01 & 901.989 & 4606.777 & 24.572 & 6.888 & 36.7 & 130.9 \cr\hline
15000 & 1.3e-04 & 1.5e+07 & -4.1e+01 & 0.0e+00 & 1.6e+01 & 1.1e+01 & 920.641 & 5115.021 & 14.827 & 6.354 & 62.1 & 144.9 \cr\hline
\end{tabular}

%% file: table_3_Suite.tex
\begin{tabular}{|ccc|cc|cc|} \hline
\multicolumn{3}{|c||}{Dim \& Density, \& $\kappa(M)$} & \multicolumn{2}{|c|}{\% Reduction in $\kappa$} & \multicolumn{2}{|c|}{CPU Time}\cr\cline{1-3}\cline{3-7}
  $n$ & density & $\kappa(M)$ & \Cref{Subgrad V} & \Cref{alg:GDeplustVv} & \Cref{Subgrad V} & \Cref{alg:GDeplustVv} \cr\hline
14822 & 3.3e-03 & 2.0e+06 & 3.9e+01 & 3.0e+01 & 7.850 & 283.641 \cr\hline
15439 & 1.1e-03 & 4.4e+12 & 8.0e+01 & 7.2e+01 & 8.734 & 2.185 \cr\hline
15439 & 6.5e-05 & 6.1e+09 & 4.2e+01 & 5.4e+01 & 2.283 & 2.725 \cr\hline
17361 & 1.1e-03 & 2.5e+06 & 1.3e+01 & 1.2e+01 & 1.097 & 9.086 \cr\hline
17361 & 3.4e-03 & 1.1e+09 & 5.6e+01 & 3.0e+01 & 19.675 & 9.306 \cr\hline
23052 & 2.2e-03 & 7.4e+11 & 9.0e+01 & 7.2e+01 & 503.668 & 6.933 \cr\hline
30401 & 5.1e-04 & 5.8e+03 & 9.1e+01 & 9.3e+01 & 204.718 & 14.951 \cr\hline
36441 & 4.3e-04 & 2.6e+03 & 8.5e+01 & 9.2e+01 & 459.869 & 59.632 \cr\hline
40806 & 1.2e-04 & 8.1e+01 & 4.5e-02 & 1.1e-01 & 29.901 & 16.301 \cr\hline
48962 & 2.1e-04 & 1.6e+05 & 6.2e+00 & 4.1e+00 & 22.109 & 43.123 \cr\hline
150102 & 3.2e-05 & 1.3e+07 & 4.6e-03 & 1.5e+00 & 53.450 & 864.384 \cr\hline
\end{tabular}

%% file: table_4_Rand.tex
\begin{tabular}{|ccc|cc|cc|} \hline
\multicolumn{3}{|c||}{Dim \& Density, \& $\kappa(M)$} & \multicolumn{2}{|c|}{\% Reduction in $\kappa$} & \multicolumn{2}{|c|}{CPU Time}\cr\cline{1-3}\cline{3-7}
  $n$ & density & $\kappa(M)$ & \Cref{Subgrad V} & \Cref{alg:GDeplustVv} & \Cref{Subgrad V} & \Cref{alg:GDeplustVv} \cr\hline
50000 & 3.8e-05 & 1.0e+11 & 1.7e+01 & 7.7e+00 & 244.577 & 74.798 \cr\hline
60000 & 3.2e-05 & 1.2e+11 & 1.5e+01 & 6.5e+00 & 400.467 & 103.404 \cr\hline
70000 & 2.7e-05 & 1.4e+11 & 1.4e+01 & 5.6e+00 & 492.698 & 123.284 \cr\hline
80000 & 2.4e-05 & 1.6e+11 & 1.1e+01 & 5.0e+00 & 436.111 & 198.894 \cr\hline
90000 & 2.1e-05 & 1.8e+11 & 1.0e+01 & 4.3e+00 & 550.245 & 223.299 \cr\hline
100000 & 1.9e-05 & 2.0e+11 & 1.2e+01 & 4.0e+00 & 986.187 & 250.367 \cr\hline
110000 & 1.7e-05 & 2.2e+11 & 1.1e+01 & 3.7e+00 & 1162.321 & 293.284 \cr\hline
120000 & 1.6e-05 & 2.4e+11 & 1.1e+01 & 3.4e+00 & 1172.119 & 332.887 \cr\hline
130000 & 1.5e-05 & 2.6e+11 & 8.5e+00 & 3.1e+00 & 1274.404 & 386.352 \cr\hline
140000 & 1.4e-05 & 2.8e+11 & 9.1e+00 & 2.9e+00 & 1500.467 & 480.042 \cr\hline
150000 & 1.3e-05 & 3.0e+11 & 8.7e+00 & 2.7e+00 & 1677.360 & 498.512 \cr\hline
\end{tabular}

%% file: table_5_PCG.tex
\begin{tabular}{|ccc|cc|cc|cc|cc|} \hline
\multicolumn{3}{|c||}{Dim, Density, \& $\kappa(A)$} & \multicolumn{2}{|c|}{\% Reduction in $\kappa$} & \multicolumn{2}{|c|}{PCG Iterations} &\multicolumn{2}{|c|}{PCG CPU Time} &\multicolumn{2}{|c|}{Total CPU Time}\cr\cline{1-3}\cline{4-11}
  $n$ & density & $\kappa(M)$ & \cite{doi:10.1287/opre.2022.0592} & \Cref{alg:GDeplustVv} & \cite{doi:10.1287/opre.2022.0592} & \Cref{alg:GDeplustVv} & \cite{doi:10.1287/opre.2022.0592} & \Cref{alg:GDeplustVv} & \cite{doi:10.1287/opre.2022.0592} & \Cref{alg:GDeplustVv} \cr\hline
 1000 & 3.3e-03 & 1.0e+09 & -2.6e+01 & 3.6e+01 & 115264 & 99023 & 1.14 & 0.99 & 5.33 & 1.29 \cr\hline
 1300 & 2.5e-03 & 1.3e+09 & -2.6e+01 & 6.0e+01 & 153760 & 101197 & 1.89 & 1.23 & 8.41 & 1.49 \cr\hline
 1600 & 2.0e-03 & 1.6e+09 & -4.3e+01 & -1.0e-13 & 180929 & 179469 & 2.67 & 2.67 & 12.17 & 2.71 \cr\hline
 1900 & 1.7e-03 & 1.9e+09 & -3.5e+01 & 6.9e+01 & 202276 & 121629 & 3.43 & 2.07 & 16.17 & 2.87 \cr\hline
 2200 & 1.4e-03 & 2.2e+09 & -4.9e+01 & 7.1e+01 & 245521 & 126567 & 4.82 & 2.50 & 21.54 & 3.39 \cr\hline
 2500 & 1.2e-03 & 2.5e+09 & -2.3e+01 & 2.7e+01 & 252757 & 221181 & 5.43 & 4.69 & 29.06 & 5.07 \cr\hline
 2800 & 1.1e-03 & 2.8e+09 & -4.2e+01 & 1.9e+01 & 289351 & 248376 & 6.69 & 5.74 & 33.60 & 6.12 \cr\hline
 3100 & 9.9e-04 & 3.1e+09 & -2.9e+01 & 6.6e+01 & 306195 & 174755 & 8.02 & 4.56 & 45.28 & 6.01 \cr\hline
 3400 & 9.0e-04 & 3.4e+09 & -3.3e+01 & 5.4e+01 & 325838 & 215488 & 9.07 & 6.05 & 54.35 & 7.46 \cr\hline
 3700 & 8.3e-04 & 3.7e+09 & -2.5e+01 & 4.4e+01 & 330583 & 256787 & 9.89 & 7.67 & 63.61 & 9.21 \cr\hline
 4000 & 7.6e-04 & 4.0e+09 & -4.4e+01 & 5.7e+01 & 366269 & 234646 & 11.60 & 7.40 & 65.77 & 9.21 \cr\hline
 4300 & 7.1e-04 & 4.3e+09 & -2.9e+01 & 5.4e+01 & 378370 & 252209 & 13.49 & 9.00 & 88.27 & 10.97 \cr\hline
 4600 & 6.6e-04 & 4.6e+09 & -4.2e+01 & 4.1e-13 & 408179 & 384608 & 15.52 & 14.65 & 89.55 & 14.74 \cr\hline
 4900 & 6.2e-04 & 4.9e+09 & -1.1e+01 & 8.7e+00 & 399371 & 382842 & 15.87 & 15.31 & 108.88 & 15.75 \cr\hline
 5200 & 5.8e-04 & 5.2e+09 & -3.7e+01 & -3.7e-14 & 458497 & 418282 & 19.04 & 17.32 & 115.01 & 17.43 \cr\hline
 5500 & 5.5e-04 & 5.5e+09 & -3.3e+01 & 3.9e+01 & 460157 & 340511 & 19.97 & 14.79 & 128.37 & 16.99 \cr\hline
 5800 & 5.2e-04 & 5.8e+09 & -3.2e+01 & 4.6e+01 & 469674 & 336957 & 21.25 & 15.32 & 143.80 & 17.73 \cr\hline
 6100 & 4.9e-04 & 6.1e+09 & -5.0e+01 & 4.3e+01 & 506736 & 353612 & 23.89 & 16.68 & 157.42 & 19.34 \cr\hline
 6400 & 4.7e-04 & 6.4e+09 & -2.2e+01 & 4.3e+01 & 492739 & 361695 & 23.98 & 17.59 & 166.18 & 20.56 \cr\hline
 6700 & 4.4e-04 & 6.7e+09 & -4.0e+01 & 4.1e+01 & 526574 & 378065 & 25.65 & 18.54 & 189.81 & 21.40 \cr\hline
 7000 & 4.2e-04 & 7.0e+09 & -2.9e+01 & 3.9e+01 & 533440 & 394622 & 28.16 & 20.13 & 202.50 & 22.93 \cr\hline
 7300 & 4.1e-04 & 7.3e+09 & -5.3e+01 & 2.7e+00 & 572826 & 511612 & 32.32 & 27.95 & 234.05 & 28.40 \cr\hline
 7600 & 3.9e-04 & 7.6e+09 & -2.9e+01 & -3.8e-14 & 555453 & 528584 & 31.27 & 29.77 & 232.71 & 29.89 \cr\hline
 7900 & 3.7e-04 & 7.9e+09 & -2.2e+01 & 3.5e+01 & 569653 & 434589 & 33.24 & 25.44 & 274.74 & 29.14 \cr\hline
 8200 & 3.6e-04 & 8.2e+09 & -6.6e+01 & 3.5e+01 & 614529 & 448278 & 37.09 & 26.97 & 323.07 & 30.42 \cr\hline
 8500 & 3.5e-04 & 8.5e+09 & -2.7e+01 & 3.4e+01 & 573590 & 458996 & 35.49 & 28.28 & 292.58 & 32.18 \cr\hline
 8800 & 3.3e-04 & 8.8e+09 & -3.6e+01 & 1.6e+01 & 625430 & 532398 & 40.04 & 34.12 & 322.79 & 35.85 \cr\hline
 9100 & 3.2e-04 & 9.1e+09 & -1.8e+01 & 1.5e+01 & 597815 & 541439 & 39.21 & 35.60 & 324.73 & 37.50 \cr\hline
 9400 & 3.1e-04 & 9.4e+09 & -2.0e+01 & 1.1e+01 & 635767 & 565196 & 42.77 & 37.92 & 374.03 & 39.77 \cr\hline
 9700 & 3.0e-04 & 9.7e+09 & -2.9e+01 & 3.1e+01 & 639683 & 508698 & 44.28 & 35.17 & 391.15 & 40.05 \cr\hline
10000 & 2.9e-04 & 1.0e+10 & -4.5e+01 & 3.1e+01 & 687294 & 521357 & 48.91 & 37.36 & 454.89 & 41.62 \cr\hline
10300 & 2.8e-04 & 1.0e+10 & -2.5e+01 & 2.9e+01 & 667536 & 538749 & 49.08 & 39.58 & 506.66 & 44.78 \cr\hline
10600 & 2.8e-04 & 1.1e+10 & -4.4e+01 & 2.9e+01 & 725311 & 548238 & 54.63 & 41.32 & 510.40 & 46.14 \cr\hline
10900 & 2.7e-04 & 1.1e+10 & -3.9e+01 & 2.6e+01 & 725604 & 566231 & 55.80 & 43.64 & 491.97 & 49.26 \cr\hline
11200 & 2.6e-04 & 1.1e+10 & -2.4e+01 & 2.8e+01 & 693639 & 569480 & 55.03 & 44.93 & 498.31 & 49.79 \cr\hline
11500 & 2.5e-04 & 1.1e+10 & -8.9e+00 & 2.4e+01 & 673910 & 597322 & 54.77 & 48.20 & 521.32 & 53.85 \cr\hline
11800 & 2.5e-04 & 1.2e+10 & -3.9e+01 & 1.7e+01 & 757749 & 627850 & 62.92 & 52.02 & 581.98 & 55.67 \cr\hline
12100 & 2.4e-04 & 1.2e+10 & -2.9e+01 & 2.6e+01 & 748235 & 601448 & 63.64 & 51.45 & 640.96 & 57.59 \cr\hline
12400 & 2.3e-04 & 1.2e+10 & -4.3e+01 & 2.5e+01 & 777435 & 614047 & 68.12 & 54.42 & 653.76 & 60.37 \cr\hline
12700 & 2.3e-04 & 1.3e+10 & -3.3e+01 & 2.5e+01 & 775871 & 620193 & 70.07 & 56.35 & 686.85 & 63.40 \cr\hline
13000 & 2.2e-04 & 1.3e+10 & -3.2e+01 & 1.1e+01 & 771101 & 681834 & 73.02 & 65.22 & 703.04 & 68.91 \cr\hline
13300 & 2.2e-04 & 1.3e+10 & -2.7e+01 & 2.4e+01 & 766793 & 648685 & 75.93 & 64.69 & 727.64 & 71.98 \cr\hline
13600 & 2.1e-04 & 1.4e+10 & -3.5e+01 & 1.0e+01 & 814120 & 708483 & 82.37 & 72.15 & 752.62 & 75.16 \cr\hline
13900 & 2.1e-04 & 1.4e+10 & -5.5e+01 & 2.3e+01 & 863835 & 662119 & 89.35 & 68.57 & 824.58 & 76.07 \cr\hline
14200 & 2.0e-04 & 1.4e+10 & -4.4e+01 & 2.3e+01 & 858674 & 673736 & 90.12 & 71.18 & 887.64 & 78.25 \cr\hline
14500 & 2.0e-04 & 1.5e+10 & -4.3e+01 & 9.4e+00 & 872710 & 741010 & 94.29 & 80.77 & 869.02 & 84.41 \cr\hline
14800 & 2.0e-04 & 1.5e+10 & -3.0e+01 & 8.5e+00 & 836054 & 747797 & 91.59 & 82.12 & 874.94 & 84.96 \cr\hline
\end{tabular}

%% file: table_6_Suite.tex
\begin{tabular}{|ccc|cc|cc|cc|cc|} \hline
\multicolumn{3}{|c||}{Dim \& Density, \& $\kappa(A)$} & \multicolumn{2}{|c|}{Ratio of Condition Numbers} & \multicolumn{2}{|c|}{CPU Time for Prec} & \multicolumn{2}{|c|}{LSQR Iterations} & \multicolumn{2}{|c|}{Total CPU Time}\cr\cline{1-3}\cline{4-11}
  $n$ & density & $\kappa(A)$ & $\kappa(S)/\kappa(T)$ & $\omega(S)/\omega(T)$ & \Cref{alg:Two-Sided} & \cite{doi:10.1287/opre.2022.0592} & \Cref{alg:Two-Sided} & \cite{doi:10.1287/opre.2022.0592} & \Cref{alg:Two-Sided} & \cite{doi:10.1287/opre.2022.0592} \cr\hline
  300 & 6.6e-03 & 5.2e+03 & 1.3e+00 & 1.0e+00 & 0.006 & 23.358 &    5 &    4 & 0.015 & 23.36 \cr\hline
  130 & 6.1e-02 & 6.1e+10 & 1.1e-01 & 8.2e-01 & 0.002 & 27.943 &    9 &   35 & 0.005 & 27.95 \cr\hline
  200 & 2.0e-02 & 2.4e+03 & 1.0e+00 & 1.0e+00 & 0.000 & 19.308 &  750 &  756 & 0.022 & 19.32 \cr\hline
  104 & 9.2e-02 & 5.5e+03 & 1.6e+00 & 9.7e-01 & 0.001 & 10.056 &  301 &  288 & 0.004 & 10.06 \cr\hline
  216 & 9.3e-02 & 3.3e+04 & 8.9e-01 & 4.6e-01 & 0.002 & 309.028 &   56 &  175 & 0.004 & 309.03 \cr\hline
  115 & 9.3e-02 & 3.7e+02 & 1.5e+00 & 6.7e-01 & 0.001 & 14.324 &  267 &  363 & 0.004 & 14.33 \cr\hline
  185 & 2.8e-02 & 1.8e+05 & 5.7e+00 & 7.5e-01 & 0.003 & 40.382 &  242 &  328 & 0.006 & 40.39 \cr\hline
  216 & 1.7e-02 & 1.0e+02 & 1.3e+00 & 9.0e-01 & 0.001 & 17.672 &  183 &  215 & 0.004 & 17.67 \cr\hline
  207 & 1.3e-02 & 1.4e+08 & 3.9e-01 & 3.4e-01 & 0.003 & 19.862 &  122 &  709 & 0.006 & 19.87 \cr\hline
  137 & 2.1e-02 & 1.8e+04 & 1.2e+00 & 7.9e-01 & 0.002 & 12.505 &   61 &   95 & 0.003 & 12.51 \cr\hline
  225 & 2.6e-02 & 7.1e+06 & 2.5e+00 & 8.1e-01 & 0.003 & 146.868 &   78 &   86 & 0.004 & 146.87 \cr\hline
  198 & 1.4e-01 & 3.0e+03 & 1.3e+00 & 5.8e-01 & 0.002 & 88.701 &  631 &  973 & 0.010 & 88.71 \cr\hline
  265 & 2.5e-02 & 1.4e+03 & 1.6e+00 & 5.0e-01 & 0.002 & 38.714 &  735 & 1379 & 0.010 & 38.73 \cr\hline
  100 & 4.0e-02 & 1.5e+04 & 1.0e+00 & 9.7e-01 & 0.000 & 8.944 &  114 &  132 & 0.002 &  8.95 \cr\hline
  105 & 8.0e-02 & 7.2e+02 & 1.5e+00 & 6.6e-01 & 0.000 & 14.751 &  163 &  227 & 0.002 & 14.75 \cr\hline
  135 & 3.6e-02 & 9.2e+05 & 1.0e+00 & 2.3e-01 & 0.000 & 17.484 &  164 &  850 & 0.002 & 17.49 \cr\hline
  120 & 6.0e-02 & 4.3e+08 & 3.1e-03 & 2.7e-02 & 0.001 & 24.967 &   77 &  845 & 0.003 & 24.97 \cr\hline
  100 & 7.1e-02 & 2.4e+12 & 5.8e-02 & 2.4e-01 & 0.001 & 28.979 &   19 &  108 & 0.003 & 28.98 \cr\hline
  136 & 2.6e-02 & 2.5e+05 & 1.9e+00 & 8.0e-01 & 0.001 & 19.051 &  132 &  175 & 0.003 & 19.05 \cr\hline
  100 & 4.0e-02 & 1.3e+04 & 2.3e+00 & 9.3e-01 & 0.000 & 10.800 &  339 &  385 & 0.003 & 10.80 \cr\hline
  300 & 3.5e-02 & 8.5e+05 & 2.2e+00 & 6.0e-01 & 0.002 & 581.259 & 1195 & 1983 & 0.016 & 581.28 \cr\hline
  132 & 2.4e-02 & 4.2e+11 & 2.1e+00 & 1.0e-01 & 0.000 & 22.553 &   93 & 1356 & 0.002 & 22.56 \cr\hline
\end{tabular}

%% file: table_7_Rand.tex
\begin{tabular}{|ccc|cc|cc|cc|cc|} \hline
\multicolumn{3}{|c||}{Dim \& Density, \& $\kappa(A)$} & \multicolumn{2}{|c|}{Ratio of Condition Numbers} & \multicolumn{2}{|c|}{CPU Time for Prec} & \multicolumn{2}{|c|}{LSQR Iterations} & \multicolumn{2}{|c|}{Total CPU Time}\cr\cline{1-3}\cline{4-11}
  $n$ & density & $\kappa(A)$ & $\kappa(S)/\kappa(T)$ & $\omega(S)/\omega(T)$ & \Cref{alg:Two-Sided} & \cite{doi:10.1287/opre.2022.0592} & \Cref{alg:Two-Sided} & \cite{doi:10.1287/opre.2022.0592} & \Cref{alg:Two-Sided} & \cite{doi:10.1287/opre.2022.0592} \cr\hline
   50 & 2.5e-01 & 1.0e+03 & 1.2e+00 & 9.4e-01 & 0.007 & 16.583 &   18 &   18 &  0.02 & 16.59 \cr\hline
   60 & 2.4e-01 & 1.0e+03 & 5.2e+00 & 5.1e-01 & 0.001 & 11.914 &   46 &   67 &  0.01 & 11.92 \cr\hline
   70 & 2.3e-01 & 1.1e+03 & 1.4e+00 & 8.3e-01 & 0.001 & 17.909 &   22 &   27 &  0.00 & 17.91 \cr\hline
   80 & 2.2e-01 & 1.1e+03 & 2.0e+00 & 2.6e-01 & 0.002 & 13.848 &   51 &  191 &  0.00 & 13.85 \cr\hline
   90 & 2.2e-01 & 1.2e+03 & 2.2e+00 & 5.0e-01 & 0.002 & 18.674 &   43 &   90 &  0.00 & 18.68 \cr\hline
  100 & 2.1e-01 & 1.2e+03 & 5.2e+00 & 4.9e-01 & 0.003 & 27.400 &   52 &   72 &  0.00 & 27.40 \cr\hline
  110 & 2.1e-01 & 1.3e+03 & 2.2e+00 & 6.4e-01 & 0.003 & 45.059 &   52 &   52 &  0.00 & 45.06 \cr\hline
  120 & 2.0e-01 & 1.4e+03 & 2.0e+00 & 7.0e-01 & 0.002 & 37.248 &   63 &   57 &  0.00 & 37.25 \cr\hline
  130 & 2.0e-01 & 1.5e+03 & 7.3e+00 & 4.0e-01 & 0.002 & 42.239 &   96 &  140 &  0.00 & 42.24 \cr\hline
  140 & 1.9e-01 & 1.6e+03 & 1.1e+01 & 5.4e-01 & 0.002 & 55.521 &   54 &   83 &  0.00 & 55.52 \cr\hline
  150 & 1.9e-01 & 1.7e+03 & 2.4e+00 & 8.0e-01 & 0.001 & 102.427 &   46 &   35 &  0.00 & 102.43 \cr\hline
  160 & 1.9e-01 & 1.8e+03 & 6.2e+00 & 5.3e-01 & 0.002 & 90.146 &   66 &   81 &  0.00 & 90.15 \cr\hline
  170 & 1.9e-01 & 1.9e+03 & 2.4e+01 & 5.1e-01 & 0.003 & 101.889 &   84 &  111 &  0.01 & 101.89 \cr\hline
  180 & 1.8e-01 & 2.1e+03 & 9.3e+00 & 4.3e-01 & 0.001 & 128.101 &   90 &  109 &  0.00 & 128.10 \cr\hline
  190 & 1.8e-01 & 2.2e+03 & 8.1e+00 & 3.4e-01 & 0.001 & 116.720 &  113 &  219 &  0.00 & 116.72 \cr\hline
  200 & 1.8e-01 & 2.5e+03 & 9.1e+00 & 4.7e-01 & 0.001 & 147.824 &   94 &  124 &  0.00 & 147.83 \cr\hline
  210 & 1.8e-01 & 2.7e+03 & 1.2e+01 & 5.2e-01 & 0.002 & 191.344 &  121 &  113 &  0.00 & 191.35 \cr\hline
  220 & 1.8e-01 & 3.1e+03 & 4.6e+00 & 5.3e-01 & 0.002 & 293.214 &   94 &   75 &  0.00 & 293.22 \cr\hline
  230 & 1.7e-01 & 3.5e+03 & 2.2e+00 & 3.8e-01 & 0.001 & 175.652 &  105 &  508 &  0.00 & 175.66 \cr\hline
  240 & 1.7e-01 & 4.0e+03 & 6.1e+00 & 4.5e-01 & 0.001 & 287.572 &  129 &  116 &  0.00 & 287.57 \cr\hline
  250 & 1.7e-01 & 4.8e+03 & 1.1e+01 & 5.4e-01 & 0.002 & 433.302 &  107 &   92 &  0.00 & 433.30 \cr\hline
  260 & 1.7e-01 & 5.9e+03 & 4.0e+00 & 5.1e-01 & 0.002 & 384.679 &  139 &  166 &  0.01 & 384.68 \cr\hline
  270 & 1.7e-01 & 7.8e+03 & 3.6e+00 & 6.3e-01 & 0.002 & 431.242 &  121 &   75 &  0.00 & 431.24 \cr\hline
  280 & 1.7e-01 & 1.1e+04 & 1.5e+01 & 3.9e-01 & 0.002 & 498.242 &  153 &  298 &  0.01 & 498.25 \cr\hline
  290 & 1.7e-01 & 2.0e+04 & 1.6e+01 & 4.5e-01 & 0.002 & 750.635 &  178 &  128 &  0.01 & 750.64 \cr\hline
  300 & 1.7e-01 & 1.0e+05 & 1.0e+01 & 2.8e-01 & 0.002 & 751.572 &  197 &  264 &  0.01 & 751.58 \cr\hline
\end{tabular}

%% file: table_8_Suite.tex
\begin{tabular}{|ccc|cc|cc|cc|} \hline
\multicolumn{3}{|c||}{Dim \& Density, \& $\omega(A)$} & \multicolumn{2}{|c|}{Ratio of Omega \& Prec CPU Time} & \multicolumn{2}{|c|}{LSQR Iter} & \multicolumn{2}{|c|}{Total CPU Time}\cr\cline{1-3}\cline{4-9}
  $n$ & density & $\omega(A)$ & $\omega(S)/\omega(A)$ & \Cref{alg:Two-Sided} CPU & No Prec & \Cref{alg:Two-Sided} & No Prec & \Cref{alg:Two-Sided} \cr\hline
14214 & 1.3e-03 & 5.1e+01 & 3.2e-02 & 1.0e-01 & 100000 &   8963 & 3.63e+01 & 3.43e+00 \cr\hline
 4119 & 2.1e-03 & 4.1e+05 & 3.6e-06 & 1.7e-02 &   5461 &    946 & 4.98e-01 & 1.04e-01 \cr\hline
 7479 & 1.2e-03 & 2.8e+04 & 4.8e-05 & 3.0e-02 &   9393 &   1758 & 1.85e+00 & 3.81e-01 \cr\hline
 9271 & 1.4e-03 & 8.2e+06 & 1.7e-07 & 4.1e-02 &   1798 &    377 & 4.44e-01 & 1.37e-01 \cr\hline
 4562 & 6.3e-03 & 2.3e+02 & 7.4e-03 & 1.3e-02 &   6497 &   1070 & 1.11e+00 & 2.00e-01 \cr\hline
 3072 & 1.3e-02 & 1.9e+00 & 1.0e+00 & 9.8e-03 &   1358 &   1357 & 5.43e-01 & 6.06e-01 \cr\hline
11341 & 7.6e-04 & 6.0e+01 & 5.0e-02 & 3.9e-02 &  26647 &   3050 & 7.55e+00 & 9.09e-01 \cr\hline
 2048 & 2.9e-03 & 2.8e+00 & 9.3e-01 & 1.6e-03 &   4730 &   4187 & 2.10e-01 & 1.93e-01 \cr\hline
 8192 & 7.3e-04 & 2.8e+00 & 9.4e-01 & 6.1e-03 &  19233 &  16963 & 3.15e+00 & 2.79e+00 \cr\hline
14734 & 4.4e-04 & 2.3e+00 & 7.2e-01 & 4.2e-02 &  10631 &   6883 & 3.43e+00 & 2.26e+00 \cr\hline
 2283 & 9.2e-03 & 3.6e+01 & 5.6e-02 & 9.1e-03 &  17604 &   1305 & 1.15e+00 & 9.74e-02 \cr\hline
 2000 & 2.0e-03 & 2.5e+00 & 1.0e+00 & 7.9e-04 &   1578 &   1578 & 6.33e-02 & 6.54e-02 \cr\hline
 2000 & 5.9e-03 & 4.2e+00 & 5.9e-01 & 3.5e-03 &   9674 &   3842 & 4.76e-01 & 1.95e-01 \cr\hline
 5000 & 2.4e-03 & 4.2e+00 & 5.7e-01 & 6.9e-03 &  29362 &   7483 & 4.56e+00 & 1.15e+00 \cr\hline
 7000 & 1.7e-03 & 4.3e+00 & 5.7e-01 & 9.3e-03 &  50422 &  13747 & 9.93e+00 & 2.82e+00 \cr\hline
 7000 & 1.7e-03 & 4.2e+00 & 5.7e-01 & 9.4e-03 &  48882 &  11272 & 9.71e+00 & 2.23e+00 \cr\hline
 3175 & 8.4e-03 & 3.0e+01 & 7.0e-02 & 9.3e-03 &  16004 &   2157 & 2.06e+00 & 2.90e-01 \cr\hline
 5952 & 6.3e-04 & 3.7e+02 & 4.6e-03 & 1.8e-02 &  13984 &   5192 & 1.55e+00 & 5.96e-01 \cr\hline
13694 & 3.9e-04 & 1.8e+02 & 7.1e-03 & 8.0e-02 & 100000 &   2551 & 3.08e+01 & 8.70e-01 \cr\hline
 5832 & 9.0e-03 & 1.7e+00 & 1.0e+00 & 1.0e-02 &   2397 &   2390 & 6.25e-01 & 6.36e-01 \cr\hline
 9801 & 9.1e-04 & 1.5e+00 & 1.0e+00 & 2.6e-03 &   3132 &   3036 & 7.93e-01 & 7.72e-01 \cr\hline
 4000 & 5.5e-04 & 1.3e+09 & 7.5e-10 & 7.3e-03 &   7194 &      5 & 5.05e-01 & 8.63e-03 \cr\hline
 5940 & 2.4e-03 & 6.6e+00 & 4.6e-01 & 3.6e-02 & 100000 &  44539 & 1.74e+01 & 7.78e+00 \cr\hline
 4326 & 3.3e-03 & 1.1e+05 & 1.7e-05 & 1.4e-02 &  34107 &   2486 & 4.85e+00 & 3.69e-01 \cr\hline
 9800 & 6.0e-04 & 2.6e+00 & 9.9e-01 & 6.4e-03 &  18382 &  18177 & 4.37e+00 & 4.29e+00 \cr\hline
\end{tabular}

%% file: table_9_PCG.tex
\begin{tabular}{|cc|cc|cc|cc|} \hline
\multicolumn{2}{|c||}{Dim \& Density} & \multicolumn{2}{|c|}{Ratios in conds $(J,M)$ } & \multicolumn{2}{|c|}{$J$: $\omega$-opt of $M$} &\multicolumn{2}{|c|}{Ratios M/J}\cr\cline{1-2}\cline{2-8}
   $n$ & density &$\kappa(J)/\kappa(M)  $ & $\omega(J)/\omega(M)  $ & iters &     cpu & iters & cpu \cr\hline
  5000 & 9.40e-03 &2.98e+00 & 7.368e-01 & 19.4 & 3.927e-03 & 25.2 & 13.9 \cr\hline
 10000 & 7.44e-03 &3.31e+00 & 7.362e-01 & 18.8 & 6.698e-03 & 36.6 & 30.5 \cr\hline
 15000 & 6.02e-03 &2.64e+00 & 7.365e-01 & 17.0 & 1.418e-02 & 41.6 & 39.0 \cr\hline
 20000 & 4.57e-03 &3.56e+00 & 7.363e-01 & 20.6 & 2.861e-02 & 41.6 & 41.1 \cr\hline
 25000 & 3.47e-03 &3.71e+00 & 7.370e-01 & 20.6 & 8.006e-02 & 30.9 & 29.1 \cr\hline
 30000 & 2.42e-03 &3.52e+00 & 7.362e-01 & 20.2 & 4.327e-02 & 48.0 & 43.0 \cr\hline
 35000 & 1.56e-03 &2.77e+00 & 7.371e-01 & 16.0 & 4.174e-02 & 36.5 & 36.1 \cr\hline
 40000 & 8.93e-04 &2.89e+00 & 7.364e-01 & 18.0 & 2.535e-02 & 46.7 & 38.9 \cr\hline
 45000 & 4.15e-04 &3.93e+00 & 7.370e-01 & 21.8 & 1.820e-02 & 28.5 & 27.0 \cr\hline
 50000 & 4.12e-04 &4.12e+00 & 7.361e-01 & 24.0 & 2.679e-02 & 49.5 & 43.7 \cr\hline
\end{tabular}

%% file: table_10_PCG_Linux.tex
\begin{tabular}{|cc|cc|cc|cc|} \hline
\multicolumn{2}{|c||}{Dim \& Density} & \multicolumn{2}{|c|}{Ratios in conds $(J,M)$ } & \multicolumn{2}{|c|}{$J$: $\omega$-opt of $M$} &\multicolumn{2}{|c|}{Ratios M/J}\cr\cline{1-2}\cline{2-8}
   $n$ & density &$\kappa(J)/\kappa(M)  $ & $\omega(J)/\omega(M)  $ & iters &     cpu & iters & cpu \cr\hline
 50000 & 9.05e-03 &3.26e+00 & 7.366e-01 & 15.6 & 3.332e-01 & 48.7 & 42.5 \cr\hline
 52000 & 8.17e-03 &4.34e+00 & 7.362e-01 & 20.8 & 4.245e-01 & 51.7 & 46.9 \cr\hline
 54000 & 7.49e-03 &3.55e+00 & 7.365e-01 & 17.0 & 3.578e-01 & 47.5 & 42.4 \cr\hline
 56000 & 6.71e-03 &3.73e+00 & 7.363e-01 & 18.0 & 3.672e-01 & 57.1 & 51.0 \cr\hline
 58000 & 5.98e-03 &3.33e+00 & 7.369e-01 & 16.8 & 3.367e-01 & 38.5 & 34.1 \cr\hline
 60000 & 5.24e-03 &3.90e+00 & 7.362e-01 & 17.6 & 3.316e-01 & 62.1 & 55.4 \cr\hline
 62000 & 4.63e-03 &4.43e+00 & 7.371e-01 & 18.0 & 3.233e-01 & 33.2 & 29.7 \cr\hline
 64000 & 4.01e-03 &4.37e+00 & 7.365e-01 & 22.0 & 3.582e-01 & 38.9 & 35.5 \cr\hline
 66000 & 3.46e-03 &3.24e+00 & 7.369e-01 & 16.0 & 2.457e-01 & 38.9 & 34.4 \cr\hline
 68000 & 2.87e-03 &4.31e+00 & 7.361e-01 & 20.8 & 2.756e-01 & 59.8 & 54.0 \cr\hline
 70000 & 2.40e-03 &4.25e+00 & 7.363e-01 & 20.4 & 2.445e-01 & 46.5 & 42.0 \cr\hline
 72000 & 1.96e-03 &4.05e+00 & 7.369e-01 & 17.8 & 1.883e-01 & 36.4 & 32.1 \cr\hline
 74000 & 1.54e-03 &3.68e+00 & 7.362e-01 & 17.2 & 1.438e-01 & 64.8 & 56.9 \cr\hline
 76000 & 1.22e-03 &3.95e+00 & 7.366e-01 & 18.0 & 1.194e-01 & 43.6 & 38.4 \cr\hline
 78000 & 9.02e-04 &2.98e+00 & 7.361e-01 & 16.4 & 8.235e-02 & 72.1 & 62.2 \cr\hline
 80000 & 6.44e-04 &3.50e+00 & 7.366e-01 & 17.0 & 6.009e-02 & 43.4 & 36.6 \cr\hline
 82000 & 4.20e-04 &3.81e+00 & 7.366e-01 & 19.2 & 4.669e-02 & 40.2 & 33.9 \cr\hline
 84000 & 2.46e-04 &4.12e+00 & 7.363e-01 & 20.2 & 3.227e-02 & 46.2 & 37.8 \cr\hline
 86000 & 1.20e-04 &2.92e+00 & 7.370e-01 & 17.0 & 1.930e-02 & 35.8 & 27.8 \cr\hline
 88000 & 4.12e-05 &3.25e+00 & 7.366e-01 & 21.6 & 1.801e-02 & 36.7 & 23.2 \cr\hline
 90000 & 1.14e-05 &1.81e+00 & 7.366e-01 & 18.8 & 1.303e-02 & 38.2 & 21.9 \cr\hline
\end{tabular}